\documentclass[reqno]{amsart}
\usepackage{amssymb, amsmath}
\usepackage{natbib}
\usepackage{dsfont}
\usepackage{mathrsfs}
\usepackage{bm}
\usepackage{color}

\DeclareMathAlphabet{\mathpzc}{OT1}{pzc}{m}{it}

\bibpunct{[}{]}{;}{n}{,}{,}

\newtheorem{te}{Theorem}[section]

\newtheorem{os}[te]{Remark}
\newtheorem{prop}[te]{Proposition}

\newtheorem{ex}[te]{Example}
\newtheorem{coro}[te]{Corollary}
\numberwithin{equation}{section}

\allowdisplaybreaks

\def \l { \left( }
\def \r {\right) }
\def \ll { \left\lbrace }
\def \rr { \right\rbrace }

\newcommand{\mmm}[1]{{\textcolor{black}{#1}}}
\newcommand{\bt}[1]{{\textcolor{black}{#1}}}
\newcommand{\MMM}[1]{{\textcolor{black}{#1}}}
\newcommand{\rev}[1]{{\textcolor{black}{#1}}}
\newcommand{\MMMrev}[1]{{\textcolor{black}{#1}}}
\newcommand{\btrev}[1]{{\textcolor{black}{#1}}}
\newcommand{\btrevv}[1]{{\textcolor{black}{#1}}}
\newcommand{\mm}[1]{{\textcolor{black}{#1}}}
\newcommand{\frev}[1]{{\textcolor{black}{#1}}}
\newcommand{\lrev}[1]{{\textcolor{black}{#1}}}
\begin{document}

	\title[]{Relaxation patterns and semi-Markov dynamics}

\author{Mark M. Meerschaert}
\address{Department of Mathematics, University of Washington, C528 Padelford Hall, Box 354350, Seattle WA 98195, USA \\
\and Department of Statistics and Probability, Michigan State University}
\email{mcubed@stt.msu.edu}

				\author{Bruno Toaldo}
	\address{Dipartimento di Matematica e Applicazioni ``Renato Caccioppoli'' - Universit\`a degli studi di Napoli ``Federico II"}
		\email{bruno.toaldo@unina.it}

	
	\date{\today}
\thanks{M. M. Meerschaert was partially supported by ARO grant W911NF-15-1-0562 and NSF grants EAR-1344280 and DMS-1462156}

\begin{abstract}
Exponential relaxation to equilibrium is a typical property of physical systems, but inhomogeneities are known to distort the exponential relaxation curve, leading to a wide variety of relaxation patterns.  Power law relaxation is related to fractional derivatives in the time variable.  More general relaxation patterns are considered here, and the corresponding semi-Markov processes are studied.  Our method, based on Bernstein functions, unifies three different approaches in the literature.
\end{abstract}

	\maketitle

\tableofcontents

\section{Introduction}
Relaxation phenomena in complex systems can deviate from the traditional exponential model.  In a heterogeneous system, a linear combination of exponential curves with varying rates can lead to power law relaxation, or a variety of other forms.  Power law (Cole-Cole) relaxation and Havriliak-Negami relaxation (transitioning between power laws frequency changes) are commonly seen in complex materials, including polymers, disordered crystals, supercooled liquids, and amorphous semiconductors \citep{havriliak,rilass,scher}.  The connection between relaxation and continuous time random walk (CTRW) models is reviewed in \cite{weron12}.  In the CTRW model, particle motions $X_n$ are separated by random waiting times $W_n$, and the long-time limiting particle density solves an evolution equation that incorporates the relaxation curve.  One famous example is the fractional Fokker-Planck equation for subdiffusive particle motions in a potential well, where delays in particle motion caused by sticking or trapping with a power law distributed waiting time lead to a fractional time derivative in the evolution equation for \rev{the} particle density \cite{HLS10PRL,Magd09a,Magd09b,MBK99,SK06}.  More general waiting time distributions lead to a variety of pseudo-differential operators in time \cite{meertri} that model general relaxation patterns.

\rev{In this paper we apply} the theory of Bernstein functions to unify the three main approaches to relaxation modeling \rev{that are} exemplified by the work of Meerschaert and Scheffler \cite{meertri}, Toaldo \cite{toaldopota}, and Magdziarz and Schilling \cite{magda}.  We show that all three approaches are equivalent, and we establish the correspondence between \rev{the model evolution} equations using conjugate Bernstein functions \cite{librobern}.  We establish some properties of solutions using regular variation theory \cite{bingam,Feller}, and \rev{we apply these solutions to construct general} semi-Markov (CTRW) particle models. The forward and backward Kolmogorov equations for Markov processes on a discrete state space are generalized to the semi-Markov case, and the classification of states into transient or \rev{recurrent} is discussed.  More general evolution equations, solved by time-changed (relaxed) semigroups on a Hilbert space, are also considered.

\btrevv{To illustrate the main ideas of this paper, we briefly consider a special case.  For $0<\beta<1$, the time-fractional diffusion equation
\[\partial^\beta_t u(x,t)=\partial_x^2 u(x,t)\]
using the Caputo fractional derivative in time [40, Eq. (2.16)] is equivalent to
\[{\mathbb D}^\beta_t u(x,t)=\partial_x^2 u(x,t)+u(x,0)\frac{t^{-\beta}}{\Gamma(1-\beta)}\]
using the Riemann-Liouville  fractional derivative in time [40, Eq. (2.17)] since these two fractional derivatives are related by
$\partial_t^\beta f(t)={\mathbb D}_t^\beta f(t)-f(0){t^{-\beta}}/{\Gamma(1-\beta)}$
for all $0<\beta<1$  [40, Eq. (2.33)].  Applying ${\mathbb D}_t^{1-\beta}$ to both sides yields a third equivalent form
\[\partial_t u(x,t)={\mathbb D}_t^{1-\beta}\partial_x^2 u(x,t)\]
using a traditional first derivative on the left-hand side.  The first form is compact, the second highlights the initial condition, and the third is most useful if one wishes to add a forcing term.  Complete details of the equivalence can be found in Example \ref{Sec4Ex1}.  Our main goal in this paper is to establish and understand the corresponding equivalence for a general class of time-nonlocal diffusion equations.  The main technical difficulty is to find an appropriate operator to apply to both sides, to convert the nonlocal time operator to a first derivative in time.  It turns out that the key is to interpret the general case of these three equations in terms of conjugate Bernstein functions.}

\section{Relaxation patterns}

\subsection{Some basic facts on Bernstein functions and subordinators}
In order to go on with the results we \rev{recall basic facts from the} theory of Bernstein functions and subordinators (see \citet{bertoinb, bertoins, librobern} for {more details}). A Bernstein function $f:(0, \infty) \mapsto [0,\infty)$ is defined to be of class $C^\infty$ and such that $(-1)^{n-1}f^{(n)}(\rev{\phi})\geq 0$ for all $n \in \mathbb{N}:=\{1,2,3,\ldots\}$. A function $f$ is a Bernstein function if and only if \cite[Thm 3.2]{librobern} it can be written in the form
\begin{align}
f(\phi) \, = \, a+b\phi + \int_0^\infty \l 1-e^{-\phi s} \r \nu(ds),
\label{defbern}
\end{align}
where $a$ and $b$ are non-negative constants and $\nu(\cdot)$ is a \rev{measure} on $(0, \infty)$ such that the integrability condition
\begin{align}
\int_0^\infty (s \wedge 1) \nu(ds) < \infty
\label{intcondlev}
\end{align}
is fulfilled. The triplet $\l a,b,\nu \r$ is said to be a L\'evy triplet. An integration by parts of \eqref{defbern} yields
\begin{align}
\phi^{-1} f(\phi) \, = \, b+ \int_0^\infty e^{-\phi s}  \bar{\nu}(s) ds,
\label{intpart}
\end{align}
where $\bar{\nu}(s) = a+\nu(s, \infty)$. It follows from \cite[Corollary 3.7 (iv)]{librobern} that \eqref{intpart} is a completely monotone function, i.e., it is $C^\infty$ and such that
\begin{align}
(-1)^n\frac{d^n}{d\phi^n}\l \phi^{-1}f(\phi) \r \geq 0, \textrm{ for all } n \in \mathbb{N} \, \rev{\cup \ll 0 \rr}.
\end{align}
A particular \rev{subset of the set of} Bernstein functions is the set of special Bernstein functions. A Bernstein function $f$ is said to be special if $f^\star (\phi) = \phi /f(\phi)$ is again a Bernstein function. The function $f^\star$, which is also special, is \rev{called} the conjugate of $f$ and has the representation
\begin{align}\label{defbern*}
f^\star (\phi)\, \mmm{=\frac{\phi}{f(\phi)}}\, = \, a^\star + b^\star \phi + \int_0^\infty \l 1-e^{-\phi s} \r \nu^\star (ds),
\end{align}
where \cite[p. 93]{librobern}
\begin{align}
b^\star  \, = \,
\begin{cases}
0, \qquad & b > 0, \\
\frac{1}{a + \nu(0, \infty)}, & b =0,
\end{cases},
\; a^\star \, = \, \begin{cases}
0, \qquad &a >0, \\
\frac{1}{b + \int_0^\infty t \nu(dt)}, & a = 0.
\end{cases}
\label{asbs}
\end{align}
In what follows we will also need complete Bernstein functions. A Bernstein function $f$ is said to be complete \cite[Def 6.1]{librobern} if the density $\nu(s)$ of the L\'evy measure $\nu(ds)\,\mmm{=\nu(s)\,ds}$ appearing in \eqref{defbern} \rev{exists and} is a completely monotone function. According to \cite[Thm 6.2, (ii)]{librobern} we have that a Bernstein function $f$ is complete if and only if $\phi \mapsto \phi^{-1}f(\phi)$ is a (non-negative) Stieltjes function, i.e., a function $h:(0, \infty) \mapsto [0, \infty)$ which can be written in the form
\begin{align}\label{hDef}
h(\phi) \, = \, \frac{a}{\phi} + b + \int_0^\infty \frac{1}{\phi + s} \mathfrak{s}(ds),
\end{align}
where $\mathfrak{s}$ is a measure on $(0, \infty)$ such that
\begin{align}
\int_0^\infty (1+s)^{-1} \mathfrak{s}(ds) < \infty.
\end{align}
It is also true that $f$ is complete if and only if the conjugate $f^\star(\phi) = \phi / f(\phi)$ is complete, \MMM{and hence every complete Bernstein function is special} \cite[Prop 7.1]{librobern}.

Bernstein functions are naturally associated with subordinators which are non-decreasing L\'evy processes. The one-dimensional distributions of a subordinator form a convolution semigroup of sub-probability \rev{measures on $[0, \infty)$, i.e., a family of measures $\ll \mu_t (\cdot) \rr_{t \geq 0}$ supported on $[0, \infty)$} such that
\begin{enumerate}
\item $\mu_t\rev{ [0,\infty)} \leq 1$,
\item $\mu_t * \mu_s = \mu_{t+s}$ for all $s, t \geq 0$,
\item $ \mu_t \to \delta_0$ vaguely as $t \to 0$\MMMrev{,}
\end{enumerate}
\MMMrev{and moreover the Laplace transform}
\begin{align}
\rev{\widetilde{\mu}_t(\MMMrev{\phi}) \, = \, } \mathcal{L} \left[ \mu_t (\cdot) \right] (\phi) \MMMrev{\, = \, \int_0^\infty e^{-\phi x}\mu_t(dx)}\, = \, e^{-tf(\phi)},
\label{laplcs}
\end{align}
where $f(\phi)$ is a Bernstein function. We will denote by $\sigma^f(t)$, $t\geq 0$, the subordinator with Laplace exponent $f$. If $f$ is a special Bernstein function then the corresponding subordinator is also \rev{called} special. The following facts will be used throughout the paper. A subordinator is special if and only if \cite[Thm 10.3]{librobern} its potential measure
\begin{align}
U^{\sigma^f}(dt) \, : = \, \mathds{E} \int_0^\infty \mathds{1}_{\ll \sigma^f(s) \in dt \rr}ds \, = \, c \delta_0(dt) + u_{\sigma^f}(t) dt,
\label{potmeas}
\end{align}
for some $c \geq 0$ and some non-increasing function $u_{\sigma^f}:(0, \infty) \mapsto (0, \infty)$ satisfying $\int_0^1 u_{\rev{\sigma^f(t)}}(t) dt < \infty$.
In particular from \cite[Corollary 10.8]{librobern} we know that if $b =0$ and $\nu(0, \infty)=\infty$ then $c=b^\star =0$, $\nu^\star (0, \infty) = \infty$, and
\begin{align}
u_{\sigma^f}(t) \, = \, a^\star  + \nu^\star (t, \infty)\ \MMM{=\bar\nu^\star(t)}.
\label{213}
\end{align}

\subsection{Relaxation patterns}
Meerschaert and Scheffler \cite{meertri} develop limit theory for the continuous time random walk (CTRW) model from statistical physics.  Given an iid sequence of jumps $J_n$ and an iid sequence of waiting times $W_n$, a particle jumps to location $S_n=J_1+\cdots+J_n$ at time $T_n=W_1+\cdots+W_n$. Given a convergent triangular array of CTRW models, they show \cite[Theorem 2.1]{meertri} that the limit process is of the form $X(L^f(t))$ where $X(t)$ is the limiting L\'evy process for the random walk of jumps, time changed by the inverse subordinator
\begin{align}\label{EtDef}
L^f(t) \, = \, \inf \ll s \geq 0 : \sigma^f(s) >t \rr.
\end{align}
Kolokoltsov \cite{KoloCTRW} extended the model to a Markov process limit $X(t)$ by allowing the distribution of the jumps $J_n$ to vary in space.  In both cases (under some mild conditions, see \cite[Theorem 4.1]{meertri} and \cite[Theorem 4.2]{KoloCTRW}) the probability densities $p(x,t)$ of the CTRW limit solve a governing equation
\begin{align}\label{CTRWgov}
\mathbb{C}_f(\partial_t)p(x,t)=Ap(x,t),
\end{align}
where $A$ is the generator of the Markov semigroup, and the Caputo-like operator \MMMrev{$\mathbb{C}_f(\partial_t)$ is defined so that
\begin{align}
\mathcal{L}\left[ \mathbb{C}_f(\partial_t)u \right](s)\, = \, f(\phi) \widetilde{u} (\phi) - \phi^{-1}f(\phi) u(0)
\label{gencap}
\end{align}
where $\mathcal{L}\left[u\right](\phi)= \widetilde{u}(\phi)=\int_0^\infty e^{-\phi t}u(t)\,dt$ is the Laplace transform}, see \cite[Remark 4.8]{meertri}.  If the waiting times $W_n$ belong to the domain of attraction of a stable subordinator with Laplace exponent $f(\phi)=\phi^\beta$, then \eqref{CTRWgov} specializes to
\begin{align}\label{CTRWgovCaputo}
\partial_t^\beta p(x,t)=Ap(x,t),
\end{align}
where $\partial_t^\beta$ is the Caputo fractional derivative \cite[Eq.\ (2.16)]{FCbook}. Other choices of $f$ lead to distributed order \cite{ultraslow} and tempered \cite{Baeumer2009,Chakrabarty} fractional derivatives \rev{(see also Example \ref{exdo}).}

\rev{Let $u$ be a real-valued function on $[0, \infty)$.} Toaldo \cite[Eq.\ (2.18)]{toaldopota} introduced the operator
\begin{align}
\mathcal{D}^f u(t)\, = \,  b\frac{d}{dt}u(t) + \frac{d}{dt} \int_0^t u(s) \bar{\nu}(t-s)ds,
\label{28}
\end{align}
where $\bar{\nu}(s) = a+\nu(s, \infty)$ for a L\'evy triplet $\l a, b, \nu \r$, and where $s \mapsto \bar{\nu}(s)$ is assumed to be absolutely continuous on \btrevv{$[s, \infty)$ for any $s>0$} \rev{(a generalized Riemann-Liouville derivative)}.
The operator \eqref{28} can be regularized by subtracting an ``initial condition'', as in \cite[Remark 4.8]{meertri}, resulting in a \rev{generalization of the regularized Riemann-Liouville derivative}
\begin{align}
\mathfrak{D}_t^f u(t) \, = \,  b\frac{d}{dt}u(t) + \frac{d}{dt} \int_0^t u(s) \bar{\nu}(t-s)ds - \bar{\nu}(t)u(0).
\label{29}
\end{align}
Use \eqref{intpart} to compute the Laplace symbol of \eqref{29} as
\begin{align}\label{eq2.18mmm}
\mathcal{L} \left[ \mathfrak{D}_t^f u(t) \right] (\phi) \, = \,& b \phi \widetilde{u} (\phi) -b u(0)+\phi \mathcal{L} \left[ u \, * \bar{\nu} \right] (\phi)- \l \frac{f(\phi)}{\phi} -b \r u(0) \notag \\
= \, &b \phi \widetilde{u} (\phi) -b u(0)+ \phi \widetilde{u} (\phi)  \l \frac{f(\phi)}{\phi}-b \r - \l \frac{f(\phi)}{\phi} -b \r u(0) \notag \\
= \, & f(\phi) \widetilde{u}(\phi) - \phi^{-1}f(\phi) u(0) .
\end{align}
This shows that \eqref{gencap} and \eqref{29} are the same operator at least for \rev{exponentially bounded continuously differentiable functions $u$}: indeed the Laplace transforms agree and furthermore $t \mapsto \mathfrak{D}^f_t u(t)$ is a continuous function since by \cite[Proposition 2.7]{toaldopota} we can write
\begin{align}
\mathfrak{D}^f_t u(t) \, =  \, &b\frac{d}{dt}u(t) + \frac{d}{dt} \int_0^t u(s) \bar{\nu}(t-s)ds - \bar{\nu}(t)u(0)\\
= \, & b \frac{d}{dt}u(t) + \int_0^t u^\prime (s) \, \bar{\nu}(t-s) \, ds
\label{gencaputo}
\end{align}
and therefore $\mathfrak{D}^f_t u(t)$ is continuous, since $u^\prime$ and $\bar{\nu}$ are continuous, hence also $u^\prime * \bar{\nu}$. Hence \eqref{29} provides an explicit definition of the operator $\mathbb{C}_f \rev{ \l \partial_t \r}$ in \eqref{CTRWgov}. \rev{Observe that \eqref{gencaputo} is a generalization of the classical Dzerbayshan-Caputo derivative (according to \cite[Definition 2.4]{toaldopota}).}

A third approach was adopted in \citet{magda}: the authors pointed out that the distribution (one-dimensional marginal) of $B \l L^f(t) \r$, a special case of the CTRW scaling limit where $B$ is a Brownian motion, is the fundamental solution to the generalized diffusion equation
\begin{align}\label{MSapproach}
\frac{\partial}{\partial t} q(x, t) \, = \, \frac{1}{2} \Phi_t \frac{\partial^2}{\partial x^2} q(x, t), \qquad x \in \mathbb{R}, t>0,
\end{align}
where $\Phi_t$ is the integro-differential operator
\begin{align}
\Phi_tu(t) \, = \, \frac{d}{dt} \int_0^t u(s) M(t-s) ds,
\label{21schi}
\end{align}
for a kernel $M(t)$ such that
\begin{align}\label{Mdef}
\mathcal{L} \left[ M(t) \right] (\phi) \, = \, \frac{1}{f(\phi)}.
\end{align}
A special case of \eqref{MSapproach} called the fractional Fokker-Planck equation,  with $M(s)=s^{-\alpha}/\Gamma(1-\alpha)$ for some $0<\alpha<1$, was introduced by Metzler et al. \cite{MBK99} in the physics literature, see also Henry et al. \cite{HLS10PRL}. The extension to a general waiting time distribution, and hence a general time-convolution operator $\Phi_t$, was pioneered by Sokolov and J. Klafter \cite{SK06} in the context of statistical physics, see also Magdziarz \cite{Magd09b}, and in the mathematical literature by Magdziarz \cite{Magd09a}. \rev{The} form \eqref{MSapproach} of the CTRW limit equation is needed when one wants to add a source/sink term with the natural units of $x/t$, see Baeumer et al.\ \cite{forcing} for additional discussion.

In this work we place these different approaches in a unifying framework by appealing to the theory of special Bernstein functions. Let $f$ be the special Bernstein function \eqref{defbern} with conjugate \eqref{defbern*} where $a^\star$ and $b^\star$ are given by \eqref{asbs}.
Assume that $\nu(0, \infty) = \infty$, so that $b^\star=0$ in view of \eqref{asbs}. As in \eqref{intpart}, an integration by parts in \eqref{defbern*} yields
\begin{align}
\phi^{-1} f^\star (\phi) \, = \,b^\star + \int_0^\infty e^{-\phi s} \bar{\nu}^\star (s) ds,
\end{align}
which implies that
\begin{align}\label{eq223}
\mathcal{L} \left[ \bar{\nu}^\star (t) \right] (\phi) \, = \, \phi^{-1}f^\star (\phi) \, = \, \frac{1}{f(\phi)},
\end{align}
since $f^\star(\phi) = \phi/f(\phi)$. \rev{We} may write $\Phi_t u(t)=\frac{d}{dt}\Psi_t u(t)$, \rev{where}
\begin{align}
\Psi_t u(t) \, = \, \int_0^t u(s) M(t-s) ds ,
\label{21schimmm}
\end{align}
and ${\mathcal L}[\Psi_t u(t)]=f(\phi)^{-1}\widetilde u(\phi)$ for continuously differentiable functions $u$.
Hence \rev{by \eqref{Mdef} and \eqref{eq223} the operator} \eqref{21schi} is related to the conjugate Bernstein function $f^\star$, while \eqref{gencap} and \eqref{28} are related to the Bernstein function $f$.  In particular, \rev{if $b=0$, then} ${\mathcal L}[\mathcal{D}_{\btrev{t}}^f u(t)]=f(\phi)\widetilde u(\phi)$ and ${\mathcal L}[\Psi_tu(t)]=f(\phi)^{-1}\widetilde u(\phi)$, so that $\mathcal{D}^f \Psi_t u(t)=\Psi_t \mathcal{D}^f u(t)$ for sufficiently smooth functions $u$. \rev{This shows} that the operator \eqref{21schimmm} is \mmm{the inverse} of the operator \eqref{28} of Toaldo \cite{toaldopota}, when $b=0$ and $M(t)=\bar\nu^\star(t)$. Then \rev{$\Phi_t \mathcal{D}^f = \frac{d}{dt} \Psi_t \mathcal{D}^f = \frac{d}{dt}$ and} heuristically \eqref{MSapproach} can be seen as the result of applying $\Phi_t$ to both sides of \eqref{CTRWgov} with $A=\frac{1}{2} \frac{\partial^2}{\partial x^2}$. This will be made precise in Theorem \ref{tegensemi}. \rev{For sufficiently smooth functions $u$ we can use \eqref{eq2.18mmm} to say also that $\Psi_t \mathfrak{D}^fu(t) = u(t) + u(0)$.} Finally, note that we also have $M(t)=u_{\sigma^f}(t)$ in view of \eqref{213}.

Next we study the eigenstructure of the operator \eqref{29}, and a corresponding property for \eqref{28}, by considering \btrev{solutions $t\mapsto q (\lambda,t) \in  C^1((0, \infty), \mathbb{R})$, continuous at zero} and exponentially bounded, to the equations
\begin{align}
\begin{cases}
\frac{d}{dt} \int_0^t q (\lambda,s) \bar{\nu}(t-s) ds - \bar{\nu}(t) q(\lambda,0) \, = \, \lambda q(\lambda,t), \qquad t \, >\,  0, \\
q(\lambda,0) = 1,
\end{cases}
\label{problrel}
\end{align}
and
\begin{align}
\begin{cases}
\frac{d}{dt} q(\lambda,t) \, = \, \lambda \frac{d}{dt} \int_0^t q (\lambda,s) \, \bar{\nu}^\star (t-s) \, ds, \qquad t > 0, \\
q(\lambda,0) = 1,
\end{cases}
\label{problreladj}
\end{align}
where $\bar{\nu}(s) = a+\nu(s, \infty)$ and $\bar{\nu}^\star (s) = a^\star+\nu^\star(s, \infty)$.
Note that in view of the discussion above the operator on the right-hand side of \eqref{problreladj} coincides with the operator \eqref{21schi} studied in \cite{magda} if $\nu (0, \infty) =\infty$ and $M(t) = \bar{\nu}^\star (t)$.
\btrev{The prototype of our solutions is clearly the Mittag-Leffler function $E_\alpha(\lambda t^\alpha):= \sum_{k=0}^\infty (\lambda t^\alpha)^k/ \Gamma(\alpha k +1)$, $\alpha \in (0,1)$, which is the eigenfunction of the Caputo fractional derivative (e.g. \cite[p. 36]{FCbook}) and also solves \eqref{problreladj} when the operator on the right-hand side is the Riemann-Liouville fractional derivative of order $1-\alpha$ (e.g. \cite[p. 12]{mainardibook}). Furthermore it is well known that this function is continuous on $[0, \infty)$ and completely monotone (and hence $C^1((0, \infty), \mathbb{R})$) and it is not differentiable at zero (e.g. \cite[Section 3.1]{mainardimittag}).}

\MMMrev{In the next theorem, we impose the additional assumption that for some $\gamma\in  (0,2)$, $C>0$ and $r_0>0$ we have
\begin{equation}\label{OreyAssumption}
\int_0^r s^2\nu(ds)>Cr^{\gamma}\quad\text{for all $0<r<r_0$,}
\end{equation}
as in Orey \cite{orey}. For example, \eqref{OreyAssumption} holds with $\gamma=2-\alpha$ if $\nu(ds)>Cs^{-1-\alpha}ds$, for all $0<s<r_0$, which is true in all the examples discussed in this paper.  Note that Meerschaert and Scheffler \cite{meertri}, Toaldo \cite{toaldopota}, and Magdziarz and Schilling \cite{magda} all assume that $\nu(0,\infty)=\infty$, and \eqref{OreyAssumption} is not much stronger.  }

\begin{te}
\label{te11}
Let $f$ be a special Bernstein function having representation \eqref{defbern} with $b=0$ and $s \mapsto \bar{\nu}(s)$ absolutely continuous on $[s, \infty)$ for any $s>0$. Assume that \eqref{OreyAssumption} holds \frev{and that if $\gamma \in \left[ \frac{3}{2}, 2 \r $ the L\'evy measure $\nu$ has a bounded density on $[s,\infty)$ for any $s>0$}. Let $f^\star$ be the conjugate of $f$ having representation \eqref{defbern*}. Let $L^f(t)$ be the inverse \eqref{EtDef} of \rev{the subordinator} $\sigma^f(t)$ with Laplace exponent $f$. Then for any $\lambda \leq 0$ the $C^1(\btrev{(}0, \infty), \mathbb{R})$, \btrev{continuous at zero} and exponentially bounded solution to \eqref{problrel} is unique and equal to the moment generating function
\begin{align}\label{qDef}
q (\lambda, t) ={\mathbb E}[e^{\lambda L^f(t)}]
\end{align}
and furthermore:
\begin{enumerate}
\item \label{1relstar} The solution \eqref{qDef} to \eqref{problrel} is also the unique continuous and exponentially bounded solution to \eqref{problreladj};
\item \label{1rel} $ [0, \infty) \ni \theta \mapsto q (-\theta,t)$ is completely monotone for each fixed $t \geq 0$, and $q (0,t)=1$ for all $t\geq 0$;
\item \label{2rel} $t \mapsto q (\lambda, t)$ is completely monotone, for each fixed $\lambda \leq 0$, \emph{if and only if} $s \mapsto \bar{\nu}(s)$ is completely monotone;
\item \label{3rel} if $f(\phi)$ is regularly varying at $0+$ \mmm{with some index $\rho\in[0,1)$} then for all $\lambda<0$,
\begin{align}
q (\lambda, t) \sim  \frac{\bar{\nu}(t)}{\mmm{a}-\lambda} \quad\text{as $t\to \infty$,}
\label{behavte}
\end{align}
\mmm{both $t\mapsto q(\lambda,t)$ and $\bar{\nu}(t)$ vary regularly at infinity with index $-\rho$}, and
\begin{align}
\int_0^\infty  q (\lambda, t) dt = \infty \textrm{ for all } \lambda < 0.
\end{align}
\end{enumerate}
\end{te}

\begin{proof}
\mmm{First we prove that \eqref{qDef} solves \eqref{problrel}.  Since $\nu(0,\infty)=\infty$}, \cite[Thm 3.1]{meertri} implies that $L^f(t)$ has a Lebesgue density $x \mapsto l(x, t)$.  Now \cite[Eq. (3.13)]{meertri} shows that
\begin{align}
\mathcal{L} \left[ l(x, \cdot) \right] (\phi) \, = \,
&\mathcal{L} \left[ \frac{\partial}{\partial x} P \ll \sigma(x) \rev{\geq} \cdot \rr \right] (\phi) \,= \, \frac{f(\phi)}{\phi} e^{-xf(\phi)}
\label{laplivt}
\end{align}
and therefore, for $\theta >0+$, we have \cite[Corollary 3.5]{meertri}
\begin{align}
\widetilde{\widetilde{l}}\l \theta, \phi \r \, = \,  \mathcal{L} \left[ \mathcal{L} \left[ l(x, t) \right] (\phi) \right] (\theta) \, = \, \frac{f(\phi)}{\phi} \frac{1}{\theta + f(\phi)}.
\label{doulapl}
\end{align}
\btrev{By \eqref{eq2.18mmm}, with $b=0$, we have that}
\begin{align}\label{eq2.36}
\mathcal{L} \left[  \frac{d}{dt} \int_0^t q(\lambda, s) \bar{\nu}(t-s) ds - \bar{\nu}(t)  \right] (\phi) \, = \, f(\phi) \widetilde{q}(\lambda, \phi) -\frac{f(\phi)}{\phi} .
\end{align}
Taking Laplace transforms in \eqref{problrel} and solving for  $\widetilde{q}(\lambda, \phi)$ then yields
\begin{align}\label{eq233}
\mathcal{L} \left[ q(\lambda, \cdot) \right](\phi) \, = \, \widetilde{q}(\lambda, \phi) \, = \, \frac{f(\phi)}{\phi} \frac{1}{f(\phi)-\lambda}.
\end{align}
Comparing \eqref{doulapl} to \eqref{eq233} shows that the moment generating function of $L^f$ is
\begin{align}
\int_0^\infty e^{\lambda x} l(x,t)\,dx=q(\lambda, t) \, = \, \mathds{E}e^{\lambda L^f (t)}, \qquad \lambda \leq 0.
\end{align}\label{qMGF}

Now we prove that $t \mapsto \mathds{E}[e^{\lambda L^f(t)}] \in  C^1((0, \infty), \mathbb{R})$. Since \eqref{OreyAssumption} holds, it follows from Orey \cite{orey} that $t \mapsto P \l \sigma^f(x) \leq t \r$ has derivatives of all orders. Hence we have that
\begin{align}
\mathds{E}[e^{\lambda L^f(t)}] \, = \, &\int_0^\infty e^{\lambda x} \l - \frac{\partial}{\partial x} P \l \sigma^f(x) \leq t \r \r dx \notag \\
= \, & 1+\lambda \int_0^\infty e^{\lambda x} P \l \sigma^f(x) \leq t \r dx \notag \\
 = \, & \btrev{ 1+\lambda \int_0^t \int_0^\infty e^{\lambda x} \mu(s, x) dx ds,}
\label{237}
\end{align}
where $t \mapsto \mu(t, x)$ is the probability density of $\sigma^f(x)$ for any $x > 0$. Now it suffices to show that the function
\begin{align}
I(s):= \int_0^\infty e^{\lambda x} \mu(s, x) dx
\end{align}
is continuous. Since we assume \eqref{OreyAssumption} we have that the density $\mu(s,x)$ can be represented via the inversion formula
\begin{align}
\mu(s,x) \, = \, (2\pi)^{-1}\int_{\mathbb{R}} e^{-i\xi s} e^{-x\varphi(\xi)} d\xi
\end{align}
where $\varphi(\xi) = f(-i\xi)$ is the characteristic exponent of the L\'evy process $\sigma^f$ and further $|e^{-x \varphi (\xi)}| \leq e^{-Cx/4|\xi|^{2-\gamma}}$ for sufficiently large $|\xi|$ (see Orey \cite{orey} at the beginning of page 937). Hence we have by the dominated convergence theorem that $\mu(s,x)$ is continuous on $(s,x) \in (0,\infty) \times (0, \infty)$. \lrev{What is more, it is bounded on $(s,x) \in (0, \infty) \times (\delta, \infty)$ for every $\delta >0$. Therefore if $s_n \to s_0>0$, as $n \to \infty$, we have by the dominated convergence theorem that
\begin{align}
\int_\delta^\infty e^{\lambda x} \mu(s_n,x) \, dx \, \to \, \int_\delta^\infty e^{\lambda x} \mu(s_0,x) \, dx \text{ as } n \to \infty.
\label{convd}
\end{align}
Also, by \eqref{213}, we have for all $s \geq s^\prime>0$,
\begin{align}
\int_0^\infty \mu(s, x) dx \, = \, \bar{\nu}^\star(s) \leq \bar{\nu}^\star(s^\prime) < \infty.
\label{2}
\end{align}
Thus, given $\epsilon>0$ there is $\delta>0$ such that
\begin{align}
\int_0^\delta \mu(s_0, x) dx < \frac{\epsilon}{2}.
\label{3}
\end{align} 
We now show a similar inequality, for $\lambda <0$,
\begin{align}
\limsup_{n \to \infty} \int_0^\delta \mu(s_n,x) dx \, \leq \,& \limsup_{n \to \infty} e^{-\lambda \delta} \int_0^\delta e^{\lambda x} \mu(s_n,x) \, dx \notag \\
\leq \, & \limsup_{n \to \infty} e^{-\lambda \delta} \l \int_0^\infty \mu(s_n,x) dx - \int_\delta^\infty e^{\lambda x} \mu(s_n,x) dx \r \notag \\
 \, = & e^{-\lambda \delta} \l \int_0^\infty \mu(s_0,x) dx - \int_\delta^\infty e^{\lambda x} \mu(s_0,x) dx  \r \notag \\
= \,&  e^{-\lambda \delta} \l \int_0^\infty \l 1-e^{\lambda x} \r \mu(s_0,x) \, dx +\int_0^\delta e^{\lambda x} \mu(s_0,x) \, dx \r,
\label{step}
\end{align}
where in the first equality above we used \eqref{convd} and the continuity of  the function $u_{\sigma^f}(s):=\int_0^\infty \mu(s,x)dx$ which follows, with our assumptions on the L\'evy measure, from \cite[Theorem 5.2]{kallenberg_ren} (see also the comments following that result).
Now we take $\lambda \uparrow 0$ (use \eqref{2} to justify the dominated convergence theorem) and we get that
\begin{align}
\limsup_{n \to \infty} \int_0^\delta \mu(s_n,x) \,  dx \, \leq 	\, \int_0^\delta \mu(s_0,x) \, dx < \frac{\epsilon}{2}.
\label{4}
\end{align}
Finally the continuity follows since using \eqref{convd}, \eqref{3} and \eqref{4}, we obtain
\begin{align}
\limsup_{n \to \infty} \left| \int_0^\infty e^{\lambda x} \mu(s_n,x) \, dx \, - \, \int_0^\infty e^{\lambda x} \mu(s_0,x) dx \right| \, \leq \, \epsilon.
\end{align}}
Hence \eqref{237} is continuous on $[s_0, \infty)$ for all $s_0>0$, and therefore $t \mapsto \mathds{E}e^{\lambda L^f(t)}$ is \mm{an element of} $C^1((0, \infty), \mathbb{R})$.

Now we can write
\begin{align}
\frac{d}{dt} \btrev{\mathds{E}} e^{\lambda L^f(t)}  \, = \, \lambda \int_0^\infty e^{\lambda x} \mu(t, x) dx .
\label{provec1}
\end{align}
Then it follows from the uniqueness of the Laplace transform (e.g., see Feller \cite[Theorem 1, p. 430]{Feller}) that \eqref{qDef} is the unique $C^1(\btrev{(}0, \infty), \mathbb{R})$ and exponentially bounded solution to the problem \eqref{problrel}, which proves the first part of the theorem.

Next we prove Item (\ref{1relstar}).
\rev{For a $C^1$ and exponentially bounded solution, by \cite[Corollary 1.6.6]{abhn}, we can} take Laplace transforms in \eqref{problreladj} to get
\begin{align}
\phi \widetilde{q} (\lambda, \phi) -1 \, = \, & \lambda \phi \mathcal{L} \left[ q \, * \, \bar{\nu}^\star \right](\phi) \notag \\
= \, & \lambda \phi  \l \phi^{-1} f^\star (\phi) -b^\star \r \widetilde{q}(\lambda, \phi).
\label{adjlapl}
\end{align}
Since $f^\star (\phi) = \phi / f(\phi)$, and $b^\star = 0$ in view of \eqref{asbs}, \eqref{adjlapl} can be rewritten
\begin{align}
\widetilde{u} (\lambda, \phi) \, = \, \frac{f(\phi)/\phi}{f(\phi) - \lambda}
\label{adjlapl2}
\end{align}
which coincides with \eqref{eq233}. This proves that \eqref{qDef} is also the unique \rev{$C^1$ and exponentially bounded} solution to \eqref{problreladj}, since they have the same Laplace transform.

\mmm{Next we prove Item (\ref{1rel}). Since the assumptions imply that $\nu(0,\infty)=\infty$, the subordinator $\sigma^f(t)$ is strictly increasing \cite[Theorem \rev{21.3}]{satolevy}, and hence $ L^f(0)=0$ a.s.}  Then we also have $q (\lambda, 0) \, = \, \mathds{E} e^{\lambda L^f(0)} = 1$. \mmm{Since  $\theta \mapsto q (-\theta,t)$ is the Laplace transform of $x\mapsto l(x,t)$, it is completely monotone for each fixed $t \geq 0$.}

Next we prove Item (\ref{2rel}).  If the function $s \mapsto \bar{\nu}(s)$ is completely monotone, we have that for some measure \rev{$m(\cdot)$ on $(0, \infty)$} and some non-negative constant $a$
\begin{align}\label{eq2.41}
\bar{\nu}(s) \, = \, a+ \int_0^\infty e^{-sw} m(dw) \, = \,a+ \int_s^\infty \int_0^\infty w\,e^{-yw} \, m(dw) \, dy
\end{align}
and therefore the function
\begin{align}
y \mapsto v (y) = \int_0^\infty e^{-yw} \, w\,m(dw)
\end{align}
is the completely monotone density of the L\'evy measure $\nu(dy)$. This implies that $f$ is a complete Bernstein function.
Now \cite[Thm 6.2 (vi)]{librobern} implies that
\begin{align}
\varphi (z) \, = \, \frac{z}{z-\lambda}
\end{align}
is a complete Bernstein function for $\lambda \leq 0$, and therefore $\varphi \circ f$ is a complete Bernstein function in view of \cite[Corollary 7.9]{librobern}. Therefore we have for some measure \rev{$k(\cdot)$ on $(0, \infty)$} that
\begin{align}
\varphi \circ f (\phi) = c+d\phi+ \int_0^\infty \l 1-e^{-\phi t} \r \int_0^\infty e^{-ts }k(ds) \, dt
\label{rev}
\end{align}
and therefore, integrating by parts in \eqref{rev}, one has
\begin{align}
\frac{1}{\phi} \l \varphi \circ f \r = \, & \int_0^\infty e^{-\phi t} \l c+d\phi+ \int_t^\infty \int_0^\infty e^{-ws} \, k(ds) \, dw \, \r dt \notag \\
= \, & \int_0^\infty e^{-\phi t} \l c+d\phi+ \int_0^\infty s^{-1}e^{-st} k(ds) \r \, dt \notag \\
 = \, & \rev{d + \int_0^\infty e^{-\phi t}  \l c+\int_0^\infty s^{-1} e^{-st} k(ds) \r dt. }
\label{final}
\end{align}
\rev{The constant $d$ in \eqref{final} is equal to zero. This can be ascertained by observing that}
\[\frac 1\phi \l \varphi\circ f\r=\frac 1\phi \frac{f(\phi)}{f(\phi)-\lambda}\to d\]
as $\phi\to\infty$ by \cite[p.\ 23, Item (iv)]{librobern} and that $f\geq 0$, $f'\geq 0$, and $-\lambda\geq 0$.  Then \btrev{since
\begin{align}
\left| \frac 1\phi \frac{f(\phi)}{f(\phi)-\lambda} \right| \leq \frac{1}{\phi} \to 0
\end{align}
as $\phi \to \infty$,} it follows that $d=0$.
Since by \eqref{eq233} we also have
\begin{align}
\frac{1}{\phi} \l \varphi \circ f \r \, = \, \int_0^\infty e^{-\phi t} q(\lambda, t)  dt ,
\end{align}
and since $t\mapsto c+ \int_t^\infty \int_0^\infty e^{-ws} \, k(ds) \, dw$ is obviously continuous, it follows from the uniqueness theorem for Laplace transforms that
\begin{align}\label{259}
q(\lambda, t) \, = \,  c+ \int_0^\infty s^{-1}e^{-st} k(ds).
\end{align}
This proves that $t \mapsto q (\lambda, t)$ is completely monotone, which establishes the direct half of Item (\ref{2rel}).

\rev{Now we prove the converse implication. By assumption we have
\begin{align}
q(\lambda, t) \, = \, c + \int_0^\infty e^{-tu}\mu(du),
\end{align}
for $c \geq 0$ and a measure $\mu(\cdot)$ on $(0, \infty)$.}
 In view of \eqref{eq233} we have
\begin{align}
\frac{1}{\phi} \frac{f(\phi)}{f(\phi)-\lambda} \, = \,  \int_0^\infty e^{-\phi t}q (\lambda, t) \, dt
\end{align}
and hence
\begin{align}
G(\phi):=\frac{1}{\phi} \frac{f(\phi)}{f(\phi)-\lambda} \, = \, & \int_0^\infty e^{-\phi t} \l c+ \int_0^\infty e^{-ts} \rev{\mu(ds)} \r dt \notag \\
= \, & \frac{c}{\phi}+ \int_0^\infty \frac{1}{\phi +s}\, \rev{\mu(ds)},
\label{convst}
\end{align}
which is a Stieltjes function provided that $\int_0^\infty (1+s)^{-1}\, \rev{\mu(ds)}<\infty$.
But such an integral converges since by \rev{Item (\ref{1relstar}), continuity of $t\mapsto q(\lambda, t)$} and the fact that $q(\lambda, 0) = 1$, it must be true that
\begin{align}
1 \, = \, q(\lambda, 0) \, = \, c+\int_0^\infty \rev{\mu(ds)}
\label{258}
\end{align}
and therefore $\rev{\mu(ds)}$ is integrable. Now note that $F(\phi)=1/G(\phi)$ is a complete Bernstein function by \cite[Thm 7.3]{librobern}. Then $\phi/F(\phi)=\phi G(\phi)$ is also complete by \cite[Proposition 7.1]{librobern}.  It follows that
\[\frac 1{\phi G(\phi)}=\frac{f(\phi)-\lambda}{f(\phi)} \, = \, 1-\frac{\lambda}{f(\phi)}\]
is a Stieltjes function in view of \cite[Thm 7.3]{librobern}. Let $g(\phi) \, := \, 1-\lambda/f(\phi)$ and use \eqref{hDef} to write
\begin{align}
g(\phi) \, = \,\frac{\mathfrak{a}}{\phi}+ \mathfrak{b}+ \int_0^\infty \frac{1}{\phi +s}\, \mathfrak{k}(ds),
\end{align}
then
\begin{align}
-\lambda/f(\phi) = g(\phi) -1\, = \, \frac{\mathfrak{a}}{\phi}+ \mathfrak{b} -1+ \int_0^\infty \frac{1}{\phi +s}\, \mathfrak{k}(ds)
\label{267}
\end{align}
and since we know that $-\lambda /f(\phi)$ is non-negative (recall that $f(\phi) \geq 0$ and $\lambda\leq 0$) also \eqref{267} must be non-negative for all $\phi \in (0, \infty)$. In particular by letting $\phi \to \infty$ we deduce that $\mathfrak{b}\geq 1$. We have thus proved that $-\lambda/f(\phi)$ is a Stieltjes function. Therefore by applying again \cite[Thm 7.3]{librobern} to the Stieltjes function $-\lambda /f$ we deduce that $f(\phi)$ is a complete Bernstein function and therefore
\begin{align}
f(\phi) \, = \, & a+\int_0^\infty \l 1-e^{-\phi s} \r \nu(s)ds \textrm{ with } \nu(s) \, = \,  \int_0^\infty e^{-st} \mathfrak{m}(dt),
\label{this}
\end{align}
for some measure $\mathfrak{m}$ and some $a\geq 0$ \MMM{(since we are assuming $b=0$)}. From \eqref{this} we get that
\begin{align}
\bar{\nu}(s) \, = \, a+\int_s^\infty \nu(w) dw \, = \, a+\int_0^\infty e^{-st} \frac{\mathfrak{m}(dt)}{t},
\end{align}
and this proves that $s \mapsto \bar{\nu}(s)$ is completely monotone, which establishes the converse part of Item (\ref{2rel}).

Finally we prove Item (\ref{3rel}).  We say that a Borel measurable function $f:(0,\infty)\mapsto[0,\infty)$ varies regularly at infinity with index $\rho\in{\mathbb R}$ if
\begin{align}
\lim_{x \to \infty} \frac{f(cx)}{f(x)} \, = \, c^\rho,
\end{align}
for any $c>0$, see for example Bingham et al. \cite[p.\ 1]{bingam}.  It follows that, for any $\varepsilon>0$, for some $x_0>0$, we have \cite[Lemma VIII.8.2]{Feller}
\begin{equation}\label{Zygmund}
x^{\rho-\varepsilon}<f(x)<x^{\rho+\varepsilon}\quad\text{for all $x\geq x_0$.}
\end{equation}
If $\rho=0$, we say that $f$ is slowly varying.  If $f(1/x)$ is regularly varying at infinity with index $-\rho$, then we say that $f$ is regularly varying at zero with index $\rho$.  Suppose that $U(x)$ is a nondecreasing right-continuous function on $[0,\infty)$ with Laplace transform
\[\widetilde U(s)=\int_0^\infty e^{-s x}U(dx)\]
for all $s>0$.  The Karamata Tauberian Theorem \cite[Thm \rev{XIII.5.2}]{Feller} states that
\begin{equation}\label{Karamata}
U(x)\sim \frac{x^\rho L(x)}{\Gamma(1+\rho)}\quad\text{as $x\to\infty$} \quad\Longleftrightarrow \quad\widetilde U(s)\sim s^{-\rho}L(1/s)\quad\text{as $s\to 0$,}
\end{equation}
where $L(x)$ is slowly varying at infinity and $\rho\geq 0$.

Suppose that $f(\phi)$ varies regularly at $\rev{\phi}=0$ with index $\rho=1-\beta$ for some  $\beta \in (0,1]$.  Note that if $f$ varies regularly at zero, we must have $\rho=1-\beta$ for some $\beta \in [0,1]$ due to the L\'evy-Khintchine representation \eqref{defbern} \cite[Proposition 1.5]{bertoins}.  If \mmm{$\rho>0$}, it follows from \eqref{Zygmund} that we must have $f(0+)=0$, and hence $a=0$ in \eqref{defbern}.  \mmm{If $\rho=0$, then $a>0$ is possible, in which case $f(0+)=a$.  In either case,} from \eqref{eq233} we have
\[
\widetilde{q}(\lambda, \phi) \, = \, \frac{f(\phi)}{\phi} \frac{1}{f(\phi)-\lambda}\sim \frac{f(\phi)}{\phi} \mmm{\frac{1}{a-\lambda}}\quad\text{as $\phi\to 0+$,}
\]
where $\lambda<0$, and then it is easy to check that $\phi\mapsto\widetilde{q}(\lambda, \phi)$ varies regularly at $\phi=0+$ with index $-\beta$.  Define
\[Q(\lambda,t)=\int_0^t q(\lambda,s)\,ds\]
so that
\[\widetilde{q}(\lambda, \phi) =\int_0^\infty e^{-\phi t}Q(\lambda,dt) .\]
Apply the Karamata Tauberian Theorem to see that $t\mapsto Q(\lambda,t)$ varies regularly at infinity with index $\beta$, and furthermore that
\[Q(\lambda,t)\sim \frac {tf(1/t)}{\Gamma(1+\beta)(\mmm{a}-\lambda)}\quad\text{as $t\to\infty$.}\]
Now apply the Monotone Density Theorem \cite[Thm 1.7.2]{bingam} to see that $t\mapsto q(\lambda,t) $ varies regulary at infinity with index $\beta-1$, and furthermore $tq(\lambda,t)/Q(\lambda,t)\to\beta$ as $t\to\infty$, so that
\begin{equation}\label{qasy}
q(\lambda,t)\sim \frac {\beta f(1/t)}{\Gamma(1+\beta)(\mmm{a}-\lambda)}\quad\text{as $t\to\infty$.}
\end{equation}

Next observe that \eqref{intpart},  along with the fact that $b=0$, implies that $f(\phi)/\phi$ is the Laplace transform of $\bar\nu(t)$.  Then another application of the Karamata Tauberian Theorem shows that
\begin{equation}\label{barnuasy}
\bar\nu(t)\sim \frac {\beta f(1/t)}{\Gamma(1+\beta)}\quad\text{as $t\to\infty$.}
\end{equation}
Combining \eqref{qasy} and  \eqref{barnuasy} shows that $(\mmm{a}-\lambda) q(\lambda,t)\sim \bar\nu(t)$ as $t\to\infty$, which proves the first statement of Item (\ref{3rel}).

Finally, since $t\mapsto Q(\lambda,t)$ varies regularly at infinity with index $\beta>0$, it follows from \eqref{Zygmund} that $Q(\lambda,t)\to\infty$ as $t\to\infty$, which proves the second statement of Item (\ref{3rel}).
\end{proof}

\begin{os} \normalfont
\label{rem22}
If $s \mapsto \bar{\nu}(s)$ is completely monotone, then we showed in the proof above that
\begin{align*}
\bar{\nu}(s) \, = \, a+ \int_0^\infty e^{-sw} m(dw) \, = \, a+ \int_s^\infty \int_0^\infty w^{-1}e^{-yw} \, m(dw) \, dy
\end{align*}
and therefore the L\'evy density of $\nu$ is also completely monotone. The corresponding Bernstein function $f$ is thus a complete Bernstein function. Therefore the adjoint $f^\star$ is also complete (Proposition 7.1 in \cite{librobern}) and has a L\'evy density which is completely monotone with tail
\begin{align*}
\bar{\nu}^\star (s) \, = \,& a^\star+\int_s^\infty \int_0^\infty e^{-tw} \rev{m^\star(dw)} \, dt \notag \\
= \, &a^\star+\int_0^\infty w^{-1}e^{-sw} \, \rev{m^\star(dw)},
\end{align*}
for some measure $m$. Therefore $s \mapsto \bar{\nu}^\star(s)$ is a completely monotone function and Item (\ref{2rel}) of Theorem \ref{te11} may be restated as: the function $t \mapsto q(\lambda, t)$ is completely monotone if and only if $s \mapsto \bar{\nu}^\star (s)$ is completely monotone.  That is, $\bar \nu$ is completely monotone if and only if $\bar{\nu}^\star$ is completely monotone.
\end{os}

\begin{os} \normalfont
\mmm{It follows from \cite[Theorem 3.1]{meertri} that the inverse subordinator \eqref{EtDef} has a probability density
\[l(x,t)=\int_0^t \bar{\nu}(t-s) \mu(ds, x)\]
for any $t>0$, where $\mu(ds, x)$ is the probability distribution of the subordinator $\sigma^f(x)$ with Laplace symbol \eqref{defbern}, and $\bar{\nu}(s) = a+\nu(s, \infty)$.  It follows that we can also write
\begin{align}
q(\lambda, t) \, = \, \int_0^t \int_0^\infty e^{\lambda x} \bar{\nu}(t-s) \mu(ds, x) \,  dx
\end{align}
in view of \eqref{qDef}.}
\end{os}

Generalized relaxation equations and patterns have been also examined in \cite{kochudo, kochuphys, kochu, weron, weron12}. Kochubei \cite{kochu} considered operators similar to that appearing in \eqref{problrel} but with different kernels of convolution. By making assumptions on the Laplace transform of such kernels he determined sufficient conditions for the complete monotonicity of the solution.  In \cite{kochudo, kochuphys} he also studied distributed-order relaxation patterns, i.e., the solution to
\begin{align}
\int_0^1 \frac{\partial^\alpha}{\partial t^\alpha} u \, \mu(\alpha)d\alpha \, = \, \lambda u, \qquad \lambda <0,
\label{dok}
\end{align}
where $\mu$ is a non-negative continuous function on $[0,1]$. He pointed out that in this case the relaxation pattern is completely monotone. Observe that \eqref{dok} is a particular case of \eqref{problreladj} (see \cite{toaldodo} for details on this point).  \mmm{An important application of \eqref{dok} is to ultraslow relaxation where $f(\phi)$ is slowly varying at $\phi=0$, see \cite{ultraslow} for more details.}

\begin{os} \normalfont
\mmm{The proof of \cite[Theorem 3.9]{ultraslow} provides a partial converse of Item (\ref{3rel}) in Theorem \ref{te11} in the case of ultraslow diffusion.  If the tail of the L\'evy measure is of the form
\[\bar\nu(t)=\int_0^1 t^{-\eta}p(\eta)d\eta,\]
where $p$ varies regularly at zero with some index $\alpha>-1$, then $\bar\nu(t)$ is slowly varying \cite[Lemma 3.1]{ultraslow}, and then it follows that the Laplace symbol
\[\phi \mapsto f(\phi)=\int_0^1 \Gamma(1-\eta)\, \rev{\phi^\eta} \, p(\eta)d\eta\]
is also slowly varying \cite[Eq.\ (3.18)]{ultraslow}. }
\end{os}

\subsection{Time-changed processes}
Theorem \ref{te11} and the discussion above suggests how the approaches of Meerschaert and Scheffler \cite{meertri}, Toaldo \cite{toaldopota}, and Magdziarz and Schilling \cite{magda} may be rearranged under a unifying framework, by resorting to special Bernstein functions. Now we extend the equations \eqref{problrel} and \eqref{problreladj} to a more general form.  Suppose that $A$ is a self-adjoint, dissipative operator that generates a $C_0$-semigroup of operators $T_t$ on the \btrev{(complex)} Hilbert space $\l \mathfrak{H}, \langle \cdot, \cdot \rangle \r$, and consider the generalized abstract Cauchy problem
\begin{align}
\frac{d}{dt} \int_0^t g(s) \, \bar{\nu}(t-s) ds \, - \bar{\nu}(t) g(0)  \, = \, A g(t), \qquad \btrev{t > 0},
\label{general}
\end{align}
or equivalently (as we will show in Theorem \ref{tegensemi})
\begin{align}
\frac{d}{dt}  g(t) \,  = \, \frac{d}{dt} \int_0^t A g( s) \, \bar{\nu}^\star(t-s) \, ds, \qquad \btrev{t>0}.
\label{generaladj}
\end{align}
\rev{Observe that if $A = \lambda \leq 0$ (then $g: [0, \infty) \mapsto \mathbb{R}$) the equations \eqref{general} and \eqref{generaladj} reduce to that studied in Theorem \ref{te11}. In Theorem \ref{tegensemi} below we will investigate solutions to \eqref{general} and \eqref{generaladj}, i.e., functions of the form $g:[0, \infty) \mapsto \mathfrak{H}$ with $g \in  C^1\l \btrev{(}0, \infty), \mathfrak{H} \r $, \btrev{$g(t)$ continuous at zero}, $g(t) \in \text{Dom}(A)$ for any $t \geq 0$ and such that \eqref{general} and \eqref{generaladj} are true.}
If $A=\frac 12 \frac{\partial^2}{\partial x^2}$ then \eqref{generaladj} with $g(t)=q(x,t)$ reduces to the equation \eqref{MSapproach} in Magdziarz and Schilling \cite{magda}.

\MMM{We follow Kolokoltsov \cite[Section 1.9]{kolokoltsov} and Schilling et al.\ \cite[Chapter 12]{librobern} for the basic theory of semigroups and generators.  See Jacob \cite[Chapter 2]{jacob1} or \cite[Chapter 11]{librobern} for a nice summary of the classical theory of linear self-adjoint operators on Hilbert spaces.  By the definition of a $C_0$-semigroup we have for all $u \in \mathfrak{H}$ that}
\begin{enumerate}
\item $T_0 u = u$
\item $ T_t T_s u = T_{t+s}u  $
\item $ \lim_{t \to 0} \left\| T_t u - u \right\|_{\mathfrak{H}} = 0 $.
\end{enumerate}
Note that since $A$ is a self-adjoint generator and it is dissipative we have that the spectrum is non-positive, i.e., for any $u \in \textrm{Dom}(A)$ we have $\langle Au,u \rangle \leq 0$ \cite[Proposition 11.2 and formula (11.4)]{librobern}, and we can apply the spectral theorem \cite[Thm 11.4]{librobern}. Therefore we know that there exists an orthogonal projection-valued measure
\begin{align}
E (B) : = \int_B E(d\lambda)
\end{align}
for Borel sets $B\subseteq\mathbb{R}$, supported on the spectrum of $A$ and therefore in this case on \mmm{a subset of} $(-\infty, 0]$, such that given a function
\begin{align}
\Xi : (-\infty, 0] \mapsto \mathbb{R}
\label{defxi}
\end{align}
we may write \mmm{\cite[Eq.\ (11.10)]{librobern}}
\begin{align}
\Xi(A)u \, = \, \int_{(-\infty,0]} \Xi(\lambda) E(d\lambda)\bt{u},
\label{specres}
\end{align}
for $u \in \textrm{Dom}(\Xi (A))$, \mmm{where by \cite[Eq.\ (11.11)]{librobern} we have}
\begin{align}
\textrm{Dom}(\Xi (A)) \, = \, \ll u \in \mathfrak{H}: \int_{(-\infty,0]} | \Xi(\lambda)|^2 \langle E(d\lambda)u, u \rangle < \infty \rr.
\end{align}
Therefore given any $u \in \mmm{\mathfrak{H}}$ we can write
\begin{align}\label{eq2.106}
\mmm{T_tu=}\, e^{At} u \, = \, \int_{(-\infty, 0]} e^{\lambda t} E (d\lambda )u.
\end{align}
Then for all $u \in  \mathfrak{H}$ we have
\begin{align}
\left\| T_tu \right\|_{\mathfrak{H}} \, = \, \left\| \int_{(-\infty,0]} e^{\lambda t} E(d\lambda ) u \right\|_{\mathfrak{H}} \, \leq \, \left\| \int_{(-\infty,0]}E(d\lambda) u \right\|_{\mathfrak{H}} \, = \, \left\| u \right\|_{\mathfrak{H}},
\label{contract}
\end{align}
so that $T_t$ is a contraction semigroup \mmm{(e.g., see \cite[Example 11.5]{librobern})}.
Since $T_t$ is a $C_0$-semigroup we also know that for $u \in \textrm{Dom}(A)$ we have \cite[Thm 1.9.1]{kolokoltsov}
\begin{align}
\frac{d}{dt} T_tu \, = \, A T_tu \, = \, T_tAu,
\label{commutat}
\end{align}
and that the map $t \mapsto T_tu$ is the unique classical solution to the abstract Cauchy problem \cite[Proposition 6.2]{engelnagel}
\begin{align}
\begin{cases}
\frac{d}{dt} g(t) \, = \, Ag(t), \\
g(0) = u.
\end{cases}
\end{align}

\begin{te}
\label{tegensemi}
Let $f$ and $q(\lambda, t)$ be as in Theorem \ref{te11} \frev{under the same assumptions on $\nu$}. Let $x \mapsto l(x, t)$ be the density of inverse process \eqref{EtDef} of the subordinator $\sigma^f(t)$ with Laplace symbol $f$. Let $T_t$ be a $C_0$-semigroup on the Hilbert space $\l \mathfrak{H}, \langle \cdot, \cdot \rangle \r$ whose generator $A$ is self-adjoint and dissipative. The unique $C^1 \l \btrev{(}0, \infty), \mathfrak{H} \r$, \btrev{continuous at zero} and exponentially bounded solution to \eqref{general}, subject to $g(0) =u\in \textrm{Dom}(A)$ coincides with the $C^1 \l \btrev{(}0, \infty), \mathfrak{H} \r$, \btrev{continuous at zero} and exponentially bounded solution of \eqref{generaladj}. This solution is the function $q(A, t)u$ defined in the sense of \eqref{specres} for all $u \in \mathfrak{H}$, and we also have
\begin{align}
q(A, t)u \, = \, \int_0^\infty T_su \, l(s, t) \, ds,
\end{align}
a Bochner integral on $\mathfrak{H}$.
\end{te}

\begin{proof}
Using the ``functional calculus'' approach introduced above we define
\begin{align}
 q (A, t)u \, = \, \int_{(-\infty,0]} q(\lambda, t) E(d\lambda) u.
 \label{pcchio}
 \end{align}
Now we recall from Theorem \ref{te11} that the function $[0, \infty) \ni \theta \mapsto  q (-\theta, t)$ is completely monotone and may be written as the Laplace transform of the density $x \mapsto l(x, t)$ of the inverse process \mmm{\eqref{EtDef}} of the subordinator $\sigma^f(t)$ with Laplace symbol $f$.

Therefore \eqref{pcchio} becomes
\begin{align}
 q (A, t) u \, = \, & \int_{(-\infty,0]} q(\lambda, t) E(d\lambda) u \notag \\
 = \, & \int_{(-\infty,0]} \int_0^\infty e^{\lambda s} \,  l(s, t) \, ds \, E(d\lambda) u \notag \\
 = \, &\int_0^\infty  \int_{(-\infty,0]}  e^{\lambda s} \, E(d\lambda) u \,  l(s, t) \, ds \notag \\
 = \, & \int_0^\infty T_su \, l(s, t) ds
 \label{timech}
\end{align}
by \eqref{eq2.106} and the Fubini-Tonelli Theorem \btrev{used under the scalar product $\langle \cdot, \cdot \rangle = \left\| \cdot \right\|_\mathfrak{H}^2$, by a simple polarization argument.}

Note that \eqref{timech} holds for any function $ u \in \mathfrak{H}$: indeed \mmm{\cite[Example 11.5]{librobern}}
\begin{align}
\left\| q(A, t) u \right\|_{\mathfrak{H}} \, = \, \left\| \int_0^\infty T_su \, l(s, t) ds  \right\|_{\mathfrak{H}} \, \leq \, \int_0^\infty \left\| T_s u \right\|_{\mathfrak{H}} l(s, t) ds \leq \, \left\| u \right\|_{\mathfrak{H}}
\end{align}
using \eqref{contract} along with $\int_0^\infty l(s, t) ds =1$.
The fact that $q(A, t)$ maps Dom$(A)$ into itself may be ascertained by \btrev{using \cite[p. 364, formula (15)]{rudinfa} to say that $q(A,t)A \subseteq  A q(A,t)$ and then \cite[p. 364, formula (10)]{rudinfa}, together with the fact that $q(A,t)$ is a bounded operator, to say that $\text{Dom}\l q(A,t)A \r = \text{Dom}(A) \subseteq \text{Dom} \l Aq(A,t) \r = \ll u : q(A,t)u \in \text{Dom}(A)\rr$.
}

\btrev{The function $t \mapsto q(\lambda, t)$ is a monotone non-increasing function (see (2.38)) continuous on $[0, \infty)$ and continuously differentiable on $(0, \infty)$ with derivative $q^\prime (\lambda, t)$ which is bounded on $t \in [t_0, \infty)$ for any $t_0 >0$. Hence continuity and differentiability properties of $t \mapsto q(A,t)$ (in the Hilbert space topology) are direct consequences of continuity and differentiability of $q(\lambda,t)$ and of the representations
\begin{align}
& \left\| q(A,t) - q(A,s)   \right\|_{\mathfrak{H}}^2 \, = \, \int_{(-\infty,0]} \l q(\lambda, t)-q(\lambda, s)\r^2 \langle E(d\lambda)u,u \rangle,  \\
& \left\| \frac{q(A, t)u - q(A, s) u}{t-s} -q^\prime (A,t) \right\|_{\mathfrak{H}}^2 \notag \\
 = \, & \int_{(-\infty,0]} \l \frac{q(\lambda, t)-q(\lambda, s)}{t-s} -q^\prime (\lambda, t) \r^2 \langle E(d\lambda)u,u \rangle,
\label{hst}
\end{align}
taken as $s \to t$. In particular the second equality, for strictly positive $t$, shows that $(d/dt)q(A,t) = q^\prime(A,t)$ since $ q^\prime (A, t)u$ exists in $\mathfrak{H}$ for all $u \in \text{Dom}(A)$ and $t>0$:}
\begin{align}
\left\|q^\prime (A, t)u \right\|_{\mathfrak{H}} \, \leq \, \int_0^\infty \left\| AT_xu \right\|_{\mathfrak{H}} \, \mu(t, x) \, dx \, \leq \, u_{\sigma^f}(t) \left\| Au \right\|_{\mathfrak{H}} < \infty.
\end{align}

The fact that $q(A, t)$ solves \eqref{general} and \eqref{generaladj} can be ascertained as follows.
\btrev{
Since $(d/dt)q(A,t) = q^\prime (A,t)$, defined in the sense of \eqref{defxi}, then for the Caputo type operator $\mathfrak{D}_t^f$ holds that
\begin{align}
\mathfrak{D}^f_t q(A, t) \, = \,  \int_{(-\infty,0]} \mathfrak{D}_t^f q(\lambda, t)  \, E(d\lambda)u.
\end{align}}
By using \eqref{29} and Theorem \ref{te11} we have that
\begin{align}
\mathfrak{D}_t^f q(A, t)  = \,  & \int_{(-\infty, 0]} \frac{d}{dt} \int_0^t q(\lambda, s) \, \bar{\nu}(t-s) \, ds \, E(d\lambda) - \, \int_{(-\infty<, 0]} \bar{\nu}(t) q(\lambda, 0) E(d\lambda) \notag \\
 = \, &\int_{(-\infty, 0]} \l \frac{d}{dt} \int_0^t q(\lambda, s) \, \bar{\nu}(t-s) \, ds \,  - \,  \bar{\nu}(t) q(\lambda, 0) \r E(d\lambda)  \notag \\
 = \, & \int_{(-\infty,0]} \lambda q(\lambda, t) E(d\lambda) \notag \\
 = \, & A q(A,t),
 \label{sperem}
\end{align}
and this proves \eqref{general}.
The same arguments can be also applied to the function $(-\infty, 0] \ni\lambda \mapsto q^\prime (A,t)$ \btrev{since $(d/dt)q(A,t) = q^\prime(A,t)$, which is of the form \eqref{defxi},} and thus by using again Theorem \ref{te11} we get
\begin{align}
q^\prime(A, t) \, = \,& \int_{(-\infty, 0]}  q^\prime (\lambda, t) \, E(d\lambda)u \notag \\
= \, & \int_{(-\infty, 0]} \frac{d}{dt} \int_0^t \lambda q(\lambda, s) \, \bar{\nu}^\star(t-s) ds \, E(d\lambda)u \notag \\
= \, & \frac{d}{dt} \int_0^t Aq(A, s) \bar{\nu}^\star (t-s) \, ds.
\end{align}

Finally we prove uniqueness.  From \cite[Eq.\ (5.13)]{toaldopota} it follows that, for any $u\in$\,Dom$(A)$, the solution $q(t)$ of the generalized Cauchy problem
\begin{equation}\label{BTgCP}
\mathfrak{D}_t^f q=Aq,\quad q(0)=u,
\end{equation}
has Laplace transform ($t\mapsto\lambda$)
\begin{equation}\label{MMMqLT}
\widetilde q(\lambda)=(f(\lambda)-A)^{-1}\frac{f(\lambda)}{\lambda}u .
\end{equation}
In view of \eqref{29} the generalized Cauchy problem \eqref{BTgCP} is another way to write \eqref{general}, and hence \eqref{MMMqLT} also holds for any \rev{exponentially bounded} solution to \eqref{general}, for any $u\in$\,Dom$(A)$.  The remainder of the argument is due to Baeumer \cite{BBchat}.  Since $A$ generates a $C_0$-semigroup, the resolvent $(f(\lambda)-A)^{-1}$ is a bounded operator for all $f(\lambda)$ in the right half plane.  In particular $(f(\lambda)-A)^{-1} 0=0 $ and hence by the uniqueness of the Laplace transform, we have $q=0$ for initial data $u=0$. Then, given two \rev{exponentially bounded} solutions $q_1,q_2$ to \eqref{BTgCP}, their difference $q=q_1-q_2$ solves \eqref{BTgCP} with $u=0$, and hence $q_1=q_2$.  Therefore, the \rev{exponentially bounded} solution to \eqref{general} is unique.  An argument similar to \eqref{adjlapl2} shows that the \rev{exponentially bounded} solution to \eqref{generaladj} for any $u\in$\,Dom$(A)$ has the same Laplace transform \eqref{MMMqLT}, hence it is also unique.
\end{proof}

\begin{os} \normalfont
\mmm{Fractional Cauchy problems of the form \eqref{CTRWgovCaputo} with $p(x,0)=f(x)\in$ Dom$(A)$ were considered by Bazhlekova \cite{Bazhlekova} and Baeumer and Meerschaert \cite{fracCauchy}.  In this case, we have $f(\phi)=\phi^\beta$ for some $0<\beta<1$.  Distributed order fractional Cauchy problems with
\[f(\phi)=\int_0^1 \phi^\beta p(\beta)d\beta, \]
were considered by Mijena and Nane \cite{MN14} and Bazhlekova \cite{Bazh15}.  Solutions to the generalized Cauchy problem \eqref{CTRWgov}, which is equivalent to \eqref{general} or \eqref{generaladj}, were developed by Toaldo \cite{toaldopota}.}
\end{os}

\section{Semi-Markov Dynamics}
In this section we construct a semi-Markov process \mmm{\eqref{Ydef}} on a countable state space \mmm{whose dynamics are governed by the operator} equations
\begin{align}
\frac{d}{dt} \int_0^t g (s) \, \bar{\nu}(t-s) \, ds \, - \bar{\nu}(t) g (0) = \, A g (t)
\label{relaxmatrback}
\end{align}
and
\begin{align}
\mmm{\frac{d}{dt}  g (t) \,   = \, \frac{d}{dt} \int_0^t  A g (s)  \, \bar{\nu}^\star (t-s)ds}
\label{relaxmatrbackadj}
\end{align}
where $A$ is an \mmm{$|\mathcal{S}| \times |\mathcal{S}|$ matrix  (we allow a countably infinite state space $|\mathcal{S}|=\infty$)} and $g(t) = q(A, t)$ is \mmm{the operator of Thm \ref{tegensemi} defined by \eqref{pcchio} using functional calculus}.
We will work all throughout this section under the following assumptions.
\begin{enumerate}
\item[A1)] $X(t)$ is a continuous-time Markov chain with countable state-space $\mathcal{S}$, generated by $A$ and associated to the semigroup of matrices $\ll P_t \rr_{t \geq 0}$. We assume that $A$ is symmetric, and that for its elements $a_{i,j}$ it is true that $\sup \ll -a_{i,i} \rr < \infty$. The assumption $\sup \ll -a_{i,i} \rr < \infty$ implies that $X(t)$ is non explosive \cite[Thm 2.7.1]{norris}. Furthemore within such a framework it is true that $P_t$ solves the so-called Kolmogorov backward equation \cite[Thm 2.8.3]{norris}
\begin{align}
&\frac{d}{dt}P_t = AP_t,
\label{kolmback}
\end{align}
and also the forward one \cite[Thm 2.8.6]{norris}
\begin{align}
&\frac{d}{dt} P_t \, = \, P_tA,
\label{kolmforw}
\end{align}
both subject to $P_0 = \bm{1}$. If $\mathcal{S}$ is finite then $A$ is a finite matrix (hence $A$ is bounded) and the representation $P_t = e^{At}$ is true. Since we do not assume that $\mathcal{S}$ is finite (but only countable) we can use the fact that $A$ is symmetric and therefore the representation $P_t = e^{At}$ is true in the sense of \eqref{specres} which becomes in this case
\begin{align}
\bt{P_t} \, = \,  e^{At} \, = \, \sum_{j} e^{\lambda_j t} v_j v_j^\prime,
\end{align}
where $\lambda_j$ are the eigenvalues of $A$ and the $v_j$ are a orthonormal basis of eigenvectors of $A$.

\item[A2)] $Y_n$ is a (homogeneous) discrete-time Markov chain on the countable state space $\mathcal{S}$ with symmetric transition matrix $H$ and we denote by $h_{i,j}$ the elements of $H$.
\item[A3)] The r.v.'s $J_i$ are i.i.d.\ with c.d.f.\ $F_J(t) = 1-q (\lambda, t)$ \MMMrev{for some $\lambda<0$}, where $t \mapsto q (\lambda, t)$ is completely monotone by Thm \ref{te11}. Assume also that the conditions of Item (\ref{3rel}) are fulfilled.  We define $T_n \, = \, \sum_{i=1}^n J_i$ and
\begin{align}\label{Ydef}
Y(t) \, = \, Y_n, \textrm{ for } T_{n} \leq t <T_{n+1}
\end{align}
and assume that the i.i.d. r.v.'s $J_i$ are also independent from $Y_n$ and therefore $T_n$ and $Y_n$ are independent. \rev{By \MMMrev{\eqref{behavte} and \eqref{eq2.41}} we have that
\begin{align}
\lim_{t \to \infty} F_J(t) \, = \, 1-\lim_{t \to \infty} \frac{\bar{\nu}(t)}{a-\lambda} \,  = \, 1- \frac{a}{a-\lambda},
\end{align}   }
and \MMMrev{hence we will assume that} $a=0$.
\item[A4)] With $\sigma^f(t)$ we denote the subordinator with Laplace exponent \eqref{defbern} for $b=0$. Since in A3) we assumed that $t \mapsto q(\lambda, t)$ is that of Thm \ref{te11} (and is completely monotone) we must have that $\nu(0, \infty) = \infty$ and that $s \mapsto \bar{\nu}(s)$ is completely monotone. Given a subordinator $\sigma^f(t)$, $t \geq 0$, with Laplace symbol $f$ we denote the inverse process \eqref{EtDef} by $L^f(t)$, $t \geq 0$. \bt{Furthermore we assume that $f(\phi)$ is regularly varying at $0+$ for some index $\rho \in [0,1)$ and therefore by Item (\ref{3rel}) of Theorem \ref{te11} the functions $t \mapsto \bar{\nu}(t)$ and  $t \mapsto q (\lambda, t)$ are regularly varying at infinity with index $-\rho$, and $\mathds{E}J_i = \infty$ for all $i$}.
\end{enumerate}
Note that equations \eqref{relaxmatrback} and \mmm{\eqref{relaxmatrbackadj}} generalize the Kolmogorov backward equation \eqref{kolmback} to semi-Markov processes. The corresponding generalizations of \eqref{kolmforw} are
\begin{align}
\frac{d}{dt} \int_0^t g(s) \, \bar{\nu}(t-s) \, ds \, - \bar{\nu}(t) g (0) = \,  g (t) A
\label{relaxmatrforw}
\end{align}
and
\begin{align}
\frac{d}{dt}  g (t) \,   = \, \frac{d}{dt} \int_0^t  g (s) A \, \bar{\nu}^\star (t-s)ds.
\label{relaxmatrforwadj}
\end{align}

\begin{coro}
\label{terelmat}
Let $A$ be as in A1). The matrix $q (A, t)$, defined in the sense of \eqref{specres}, where $q (\lambda, t)$ is the function of Theorem \ref{te11} \frev{(under the same assumptions on $\nu$)} is the unique solution to \eqref{relaxmatrback} and \eqref{relaxmatrforw} as well as \eqref{relaxmatrbackadj} and \eqref{relaxmatrforwadj} with initial datum $g(0) = \bm{1}$ (identity matrix). Furthermore we have, using the Bochner integral, that
\begin{align}
q (A, t) \, = \, \int_0^\infty P_s \; l(s, t)ds .
\end{align}

\end{coro}

\begin{proof}
This Corollary is a direct consequence of Thm \ref{tegensemi}. \mmm{To clarify the arguments, here we provide some details where the proof becomes simpler in the present case.}
Observe that now the generator $A$ is a matrix with non-positive eigenvalues and, for an orthonormal basis $v_j$ of eigenvectors of $A$, the spectral representation \eqref{specres} here is the matrix
\begin{align}
\Xi(A) \, = \, \sum_j \Xi(\lambda_j) v_j v_j^\prime.
\label{bello}
\end{align}
The function $q(A, t)$ is therefore the matrix
\begin{align}
 q (A, t) \, = \, & \sum_j q (\lambda_j, t) v_j v_j^\prime \notag \\
   = \, & \sum_j \int_0^\infty e^{\lambda_j s} l(s, t) ds  \, v_j v_j^\prime \notag \\
    = \, & \int_0^\infty P_s \; l(s, t) ds.
    \label{tcmatr}
\end{align}
It is clear that equation \eqref{relaxmatrforw} is also satisfied since $P_tA = AP_t$. In our case this may be easily checked
\begin{align}
Aq (A, t) \, = \, & \sum_j \lambda_j v_j v_j^\prime \sum_i \int_0^\infty e^{\lambda_i s} l(s, t) ds \, v_i v_
i^\prime \notag \\
= \, & \int_0^\infty \sum_j \sum_i \lambda_j e^{\lambda_i s} v_j v_j^\prime v_i v_i^\prime \, l(s, t) ds \notag \\
= \, & \sum_j \int_0^\infty \lambda_j e^{\lambda_j s} v_j v_j^\prime l(s, t) ds \notag \\
= \, & \sum_j \int_0^\infty e^{\lambda_j s} v_j v_j^\prime l(s, t) ds \sum_i \lambda_i v_i v_i^\prime \notag \\
= \, & q (A, t) A
\label{commmatr}
\end{align}
which finishes the proof.
\end{proof}

\begin{te}
\label{secsecsec}
\bt{The process $Y(t)$ introduced in A3) is semi-Markov and} such that the $|\mathcal{S}| \times |\mathcal{S}|$ matrix with elements
\begin{align}
q_{i,j}(t) \, = \, P \ll Y(t) = j | Y(0) =i \rr
\end{align}
satisfies \eqref{relaxmatrback} and \eqref{relaxmatrforw} as well as \eqref{relaxmatrbackadj} and \eqref{relaxmatrforwadj} with initial datum $g(0) = \bm{1}$ for $A =-\lambda ( H-\bm{1})$, where $H$ is the transition matrix of the discrete-time Markov chain $Y_n$ on $\mathcal{S}$ introduced in A2). Furthermore
\begin{align}
\l q_{i,j}(t) \r_{i,j} \, = \, q (A, t) \, = \, \int_0^\infty e^{-\lambda (H-\bm{1})s} l(s, t) \, ds,
\end{align}
where $q (A, t)$ is the function of Thm \ref{tegensemi} and $s \mapsto l(s, t)$ is the density of the process $L^f(t)$ in A4).
\end{te}

\begin{proof}
First note that since $Y(t) = Y_n$, for $T_n \leq  t< T_{n+1}$, we have that $Y(t) = Y_{N^{\bar{\nu}}(t)}$ where
\begin{align}\label{DefNbarnut}
N^{\bar{\nu}}(t) \, = \, \max \ll n \in \mathbb{N}: T_n \leq t  \rr,
\end{align}
and therefore $Y(t)$ is a semi-Markov process since
\begin{align}
P \l Y_n = j, J_n \leq t  | \l Y_0, T_0 \r  \cdots \l Y_{n-1}, T_{n-1} \r \r \, = \, P \l Y_n = j | Y_{n-1} = i \r \l 1-q (\lambda,t) \r,
\end{align}
due to the independence between the r.v.'s $J_i$ and the chain $Y_n$.
Therefore the $q_{i,j}$ satisfy the (backward) renewal equation \cite[Chapter 10, formula (5.5)]{cinlar}
\begin{align}
q_{i,j}(t) \, = \, & q(\lambda, t) \delta_{i,j} + \sum_{l \in \mathcal{S}} h_{i,l} \int_0^t q_{l,j} (s) \, \mathfrak{f}_J(t-s) ds
\label{reneq}
\end{align}
where $h_{i,j}$ are the elements of the symmetric transition matrix $H$ of the discrete-time chain \rev{and
\begin{align}
\mathfrak{f}_J(t) = \frac{d}{dt} \l 1- q(\lambda, t) \r.
\end{align}  }

Next we prove that $t \mapsto q_{i,j}(t)$ is continuous. \mmm{Since $t \mapsto q(\lambda, t)$ is completely monotone under A3), we can write} \begin{align}
q(\lambda, t) \, = \, \int_0^\infty e^{-t w} \frac{m(dw)}{w}
\end{align}
for some measure $m(dw)$, \rev{and hence we also have
\begin{align}
\mathfrak{f}_J(t)=\int_0^\infty e^{-tw}m(dw)
\label{cmdens}
 \end{align}
and this means that also $t \mapsto \mathfrak{f}_J (t)$ is completely monotone.
Then \eqref{reneq} yields for any $i,j \in \mathcal{S}$ and $t,h>0$
\begin{align}
& \left| q_{i,j}(t)-q_{i,j}(t+h) \right| \notag \\
\leq \, &\left| q(\lambda, t) - q(\lambda, t+h) \right| + \left| \int_0^t \sum_{l \in \mathcal{S}} h_{i,l}q_{l,j}(s) \l \mathfrak{f}_J( t-s) - \mathfrak{f}_J( t+h-s) \r ds \right| \notag \\
&+ \left|  \int_t^{t+h} \sum_{l \in \mathcal{S}} h_{i,l}q_{l,j}(s) \mathfrak{f}_J(t+h-s) ds \right| \notag \\
\leq \, &   q(\lambda, t) - q(\lambda, t+h)  +  \int_0^t \l \mathfrak{f}_J( t-s) - \mathfrak{f}_J( t+h-s) \r ds + \int_0^h  \mathfrak{f}_J  (s) ds
\label{contq}
\end{align}
where we used that $\left| \sum_{l \in \mathcal{S}} h_{i,l}q_{l,j}(s) \right| \leq \sum_{l \in \mathcal{S}} h_{i,l} =1$. Then the first and the last terms in \eqref{contq} go to zero as $h\to 0+$ since $t \mapsto q(\lambda, t)$ and $t \mapsto \mathfrak{f}_J(t)$ are completely monotone functions. For the integral we can use \eqref{cmdens} to say that
\begin{align}
\l \mathfrak{f}_J( t-s) - \mathfrak{f}_J(t+h-s) \r \, \leq \, \mathfrak{f}_J( t-s)
\end{align}
and $\int_0^t \mathfrak{f}_J( t-s) ds < \infty$ for all $t>0$. Therefore the integral in \eqref{contq} goes to zero as $h \to 0+$ by dominated convergence theorem. For $h<0$ the argument is similar.  Hence $t \mapsto q_{i,j}(t)$ is continuous.
}

Therefore the $q_{i,j}(t)$, $i,j \in \mathcal{S}$, are the unique continuous functions whose Laplace transforms satisfy (we here use \mmm{\eqref{eq2.36} and} \eqref{reneq})
\begin{align}
\widetilde{q}_{i,j}(\phi) \, = \, \frac{f(\phi)/\phi}{f(\phi)-\lambda}\delta_{i,j}- \sum_{l \in \mathcal{S}} h_{i,l} \widetilde{q}_{l,j}(\phi) \frac{\lambda}{f(\phi)-\lambda}.
\label{calc}
\end{align}
Now multiply by \mmm{$f(\phi)-\lambda$ on} both sides of \eqref{calc} and substract $q_{i,j}(0) = \delta_{i,j}$ to get
\begin{align}
f(\phi) \widetilde{q}_{i,j}(\phi)-\frac{f(\phi)}{\phi} q_{i,j}(0) \, = \, -\lambda \l \sum_{l \in \mathcal{S}} \widetilde{q}_{l,j}(\phi) h_{i,l} -\widetilde{q}_{i,j}(\phi) \r .
\label{laplincv}
\end{align}
\mmm{Now multiply by $f^\star(\phi)=\phi/f(\phi)$ to get}
\begin{align}
\phi \widetilde{q}_{i,j}(\phi)- q_{i,j}(0) \, = \, -\lambda f^\star(\phi) \l \sum_{l \in \mathcal{S}} \widetilde{q}_{l,j}(\phi) h_{i,l} -\widetilde{q}_{i,j}(\phi) \r.
\label{laplincvadj}
\end{align}
By Laplace inversion of \eqref{laplincv} and \eqref{laplincvadj} we have \mmm{using \eqref{intpart}} that $q_{i,j}(t)$ satisfies for all $i,j \in \mathcal{S}$
\begin{align}
\begin{cases}
\frac{d}{dt} \int_0^t q_{i,j}(s) \, \bar{\nu}(t-s) ds -\bar{\nu}(t) q_{i,j}(0) \, = \, -\lambda \l \sum_{l \in \mathcal{S}} q_{l,j}(t)h_{i,l}-q_{i,j}(t)  \r, \\
q_{i,j}(0) = \delta_{i,j},
\end{cases}
\label{227}
\end{align}
and
\begin{align}
\begin{cases}
\frac{d}{dt}  q_{i,j}(t) \, = \, -\lambda \frac{d}{dt} \int_0^t \l \sum_{l \in \mathcal{S}} q_{l,j}(s)h_{i,l}-q_{i,j}(s)  \r \bar{\nu}^\star(t-s) ds ,\\
q_{i,j}(0) = \delta_{i,j}.
\end{cases}
\label{227adj}
\end{align}
Note that the solution to the matrix problem, for $A=-\lambda (H-\bm{1})$,
\begin{align}
\begin{cases}
\frac{d}{dt} \int_0^t q (A,s) \, \bar{\nu}(t-s) ds -\bar{\nu}(t) q(A,0) \, = \, A \mmm{q}(A,t), \\
q(A,0)=1,
\end{cases}
\end{align}
is a matrix such that each entry satisfy the backward equation \eqref{227}.
The backward equation therefore is proved.

\mmm{The forward equation follows by \eqref{commmatr}. To verify in this special case,} since $H$ is symmetric then so is $ H-\bm{1}$ and the eigenvalues of $H-\bm{1}$ are non-positive since $H$ is a transition matrix. Therefore we can write as in \eqref{tcmatr}
\begin{align}
q_{-\lambda (H-\bm{1})} (t) \, = \,  & \int_0^\infty e^{-\lambda (H-\bm{1})s} l(s, t) ds,
\end{align}
and we know that
\begin{align}
-\lambda (H-\bm{1})q (-\lambda (H-\bm{1}),t) = q (-\lambda (H-\bm{1}), t) (-\lambda) (H-\bm{1}),
\end{align}
in view of \eqref{commmatr}.
\end{proof}

If we interpret the i.i.d. r.v.'s $J_i$ as waiting times between events in some point process then the sequence $J_i$ is a renewal process and the r.v. $T_n$ is the instant of the $n$-th event. The process counting the number of events occurred up to a certain time $t$ is the counting process $N^{\bar{\nu}}(t) = \max \ll n \in \mathbb{N}: T_n \leq t \rr$. Clearly if one considers exponentially distributed waiting times then the corresponding counting process is the Poisson process. When the waiting times are Mittag-Leffler distributed, i.e., $P \ll J > t \rr = E_{\alpha}(\lambda t^\alpha)$, $\alpha \in (0,1)$, $\lambda < 0$, a particular semi-Markov model on a graph have been considered in \cite{G15, raberto}. The authors showed that the governing equation is time-fractional. In general if the i.i.d waiting times $J_i$ have finite mean $\mu_J$ then one has by \mmm{a simple argument using the} strong law of large numbers that a.s.
\begin{align}
\lim_{t \to \infty} \frac{N^{\bar{\nu}}(t)}{t} \, = \, \frac{1}{\mu_J},
\label{313}
\end{align}
and the elementary renewal theorem \cite[Proposition 1.4]{Asmussen} states that
\begin{align}
\lim_{t \to \infty} \frac{\mathds{E}N^{\bar{\nu}}(t)}{t} \, = \, \frac{1}{\mu_J}.
\label{314}
\end{align}
These facts may be interpreted as an equivalence (in the long time behaviour) between the Poisson process and a general renewal process with finite-mean waiting times. Note that \eqref{313} and \eqref{314} means that if $\mathds{E}J_i<\infty$ then as $t \to \infty$ we have, a.s.,  $N(t) \sim N^{\bar{\nu}}(t)$. This heuristically means that a renewal process with finite mean waiting times is indistinguishable after a ``long time" from the Poisson process. Therefore when you observe the process $Y(t)$ defined in A3) with sojourn times $J_n=T_{n+1}-T_n$ having finite mean then, after a transient period it behaves like the case in which $T_{n+1}-T_n$ are exponential r.v.'s.

We have here introduced a class of renewal processes associated with waiting times $J$ such that $P \ll J >t \rr = q (\lambda, t)$ and under $A3)$ we know that $\mathds{E}J = \infty$. Therefore they never behave as a Poisson process. Corollary \ref{terelmat} and Theorem \ref{secsecsec} implies that the time-changed Markov chain $X \l L^f(t) \r$, $t \geq 0$, may be equivalently constructed starting from an embedded Markov chain $Y_n$ and by inserting between jumps the heavy-tailed waiting times $J_i$.

\mmm{In equation \eqref{Ydef} we defined the renewal process $Y(t)$ that jumps to the state $Y_n$ for the underlying discrete time Markov process at the arrival time $T_n$ of the renewal process with waiting time distribution ${\mathbb P}[J>t]=q(\lambda,t)$.  Then we have $Y(t)=Y_{N^{\bar{\nu}}(t)}$, a time change using the renewal process \eqref{DefNbarnut}.  Next we show that the same process can also be constructed by a time change using the inverse subordinator \eqref{EtDef}.  }

\begin{prop}
\label{propeq}
Let $N(t)$ be a homogeneous Poisson process with rate $\theta =-\lambda$. The time-changed process $N \l L^f(t) \r$ and the process $N^{\bar{\nu}}(t)$ are the same process.  \mmm{Hence the semi-Markov process \eqref{Ydef} is the same process as the time-changed Markov chain $Y_{N \l L^f(t) \r}$.}
\end{prop}

\begin{proof}
This is a consequence of \cite[Thm 4.1]{meerpoisson} since $P \ll J_i > t \rr = q(\lambda, t)$ for any $i$ and in view of Thm \ref{te11} it is true that $q(\lambda, t) = \mathds{E}e^{\lambda L^f(t)}.$
\end{proof}

\subsection{Classification of states}
We here investigate whenever a state $i$ is transient or recurrent for the semi-Markov process $Y(t)$, by making assumptions on the embedded chain $Y_n$. We recall that a state $i$ is recurrent if
\begin{align}
P \l \textrm{the set } \ll t \geq 0 : Y(t) = i \rr \textrm{ is unbounded } \mid Y(0) = i  \r \, = \, 1,
\label{recdef}
\end{align}
and is transient if the probability in \eqref{recdef} is zero.  \mmm{Since we take the probability of a tail event, 0 and 1 are the only possibilities.}
We have the following result.

\begin{te}
\label{thmstate}
Under A2), A3), and A4) it is true that
\begin{enumerate}
\item \label{1state} If the state $i$ is recurrent for $Y_n$ then it is recurrent for $Y(t)$
\item \label{2state} If the state $i$ is transient for $Y_n$ then it is transient for $Y(t)$
\item \label{3state} $\int_0^\infty q_{i,i}(t) dt = \infty$ independently from the fact that the state $i$ is transient or recurrent.
\end{enumerate}
We recall that $q_{i,i}(t) = P \ll Y(t) = i \mid Y(0) =i \rr = P^i \ll Y(t) =i \rr$.
\end{te}

\begin{proof}
Note that in view of Proposition \ref{propeq} we can write
\begin{align}
Y(t) \, = \, Y_{N^{\bar{\nu}}(t)} \, = \, Y_{N \l L^f(t) \r} .
\end{align}
\mmm{Since $t\mapsto \sigma^f(t)$ is a.s.\ right-continuous, unbounded, and strictly increasing, it follows from the change of variable formula in Meerschaert and Straka \cite[p.\ 1707]{meerstra} that}
\begin{align}
\int_0^\infty \mathds{1}_{\ll t \geq 0 : Y(t) = i \rr} dt
\, = \,& \int_0^\infty \mathds{1}_{\ll t \geq 0 : Y_{N \l L^f(t) \r} = i \rr} \, dt \notag \\
\, = \,&  \int_0^\infty \mathds{1}_{\ll t \geq 0 : Y_{N(t)} = i  \rr} d\sigma^f(t) \notag \\
= \, & \sum_{n}  \mathds{1}_{\ll Y_n = i \rr} \sum_{\tau_n \leq t < \tau_{n+1}} \mathpzc{e}(t),
\label{levito}
\end{align}
where $\mathpzc{e}(t)$ is the Poisson point process underlying the subordinator $\sigma^f$ with characteristic measure $\nu(s)ds$. Note that if the state $i$ is recurrent for $Y_n$ then
\begin{align}
P \l Y_n = i \textrm{ for infinitely many } n \r \, = \, 1
\end{align}
and the number of summands in \eqref{levito} is a countable infinity. Furthermore the sequence
\begin{align}
\sum_{\tau_n \leq t < \tau_{n+1}} \mathpzc{e}(t)
\end{align}
is a sequence of i.i.d. r.v.'s since $\mathpzc{e}(s)$ is a Poisson point process. Therefore \eqref{levito} is the sum of a countable infinity of i.i.d. positive r.v.'s. which diverges with probability one. This proves Item \ref{1state}. If instead the state $i$ is transient, then the number of summands in \eqref{levito} is finite, and since our subordinators are here assumed to be not subject to killing, it is true that for all $0 \leq t_1 < t_2 < \infty$
\begin{align}
0< \sum_{t_1 \leq s \leq t_2} \mathpzc{e}(s)  < \infty \textrm{ a.s. }
\end{align}
and the sum \eqref{levito} is finite since it is the sum of a finite number of finite summands. This proves \ref{2state}. Observe now that
\begin{align}
\int_0^\infty q_{i,i}(t) \, dt \, = \, &\mathds{E}^i \int_0^\infty \mathds{1}_{\ll Y(t) = i \rr} \, dt \notag \\
\geq  \, & \mathds{E}^i J_1 \, = \, \infty
\end{align}
and this proves Item \ref{3state}.
\end{proof}

\mmm{It is instructive to compare Part (3) of Theorem \ref{thmstate} with the well-known characterization of transient and recurrent states in a semi-Markov process with finite mean waiting time between jumps, using the occupation measure (or $0$-potential) \cite{cinlar}.  The semi-Markov process is formed by inserting a random waiting time $J_i$ before jumping from state $Y_{i-1}$ to state $Y_i$ in the underlying Markov chain.  Hence the number of times the process returns to its starting point is not affected.  Hence the state is recurrent for the semi-Markov process if and only if it is recurrent for the underlying Markov chain. If the occupation times (waiting times) in each state have a finite mean, then the total expected occupation time in the starting state is proportional to the number of visits.  This happens if and only if $\int_0^\infty q_{i,i}(t) dt = \infty$.   However, when the waiting times between state transitions are heavy-tailed with infinite mean, the mean time spent in the starting point by the process is always infinite.}

\section{\mmm{Examples}}
In this section, we provide some practical examples, to illustrate the application of the results in this paper.

\begin{ex}\label{Sec4Ex1}\normalfont
Consider a CTRW with iid particle jumps $X_n$ independent of the iid waiting times $W_n$.  Then a particle arrives at location $S(n)=X_1+\cdots+X_n$ at time $T_n=W_1+\cdots+W_n$.  If ${\mathbb E}[X_n]=0$ and ${\mathbb E}[X_n^2]<\infty$, then $n^{-1/2}S([nt])\Rightarrow B(t)$, a Brownian motion, by Donsker's Theorem \cite[Theorem 8.7.5]{Durrett}.  If ${\mathbb P}[W_n>t]=t^{-\beta}/\Gamma(1-\beta)$ for some $0<\beta<1$, then $n^{-1/\beta}T_{[nt]}\Rightarrow \sigma^f_t$, a $\beta$-stable subordinator with Laplace symbol $f(\phi)=\phi^\beta$ \cite[Theorem 3.41]{FCbook}.  The number of jumps by time $t>0$ is $N_t=\max\{n\geq 0:T_n\leq t\}$ and a simple argument \cite[Theorem 3.2]{meerscheflimctrw} shows that this inverse process has an inverse limit $n^{-\beta}N_{nt}\Rightarrow L^f_t$ where the inverse stable subordinator $L^f_t$ is defined by \eqref{EtDef}.  Then we have
\[n^{-\beta/2}S(N_{nt})\approx (n^{\beta})^{-1/2}S(n^{\beta}\cdot n^{-\beta}N_{nt})\approx (n^{\beta})^{-1/2}S(n^{\beta}L^f_t)\approx B(L^f_t),\]
as $n\to\infty$, see \cite[Theorem 4.2]{meerscheflimctrw} for a rigorous argument.  The probability density $p(x,t)$ of the limit process $B(L^f_t)$ solves the time-fractional diffusion equation \eqref{CTRWgov} with $\mathbb{C}_f(\partial_t)=\partial_t^\beta$, a Caputo fractional derivative of order $0<\beta<1$, and $A= D\partial^2_x$ where $D={\mathbb E}[X_n^2]/2$, see \cite[Theorem 5.1]{meerscheflimctrw} and \cite[Eq. (1.8)]{FCbook}.
Since the Caputo fractional derivative \cite[Eq.\ (2.16)]{FCbook}
\begin{equation*}
\partial_t^\beta p(x,t)=\frac{1}{\Gamma(1-\beta)}\int_0^t \frac{\partial}{\partial t}p(x,t-u)u^{-\beta}du
\end{equation*}
is related to the Riemann-Liouville fractional derivative \cite[Eq.\ (2.17)]{FCbook}
\begin{equation*}
{\mathbb D}_t^\beta p(x,t)=\frac{1}{\Gamma(1-\beta)}\frac{\partial}{\partial t}\int_0^t p(x,t-u)u^{-\beta}du
\end{equation*}
by \cite[Eq. (2.33)]{FCbook}
\begin{equation}\label{RLtoCaupto}
{\mathbb D}_t^\beta p(x,t)-p(x,0)\frac{t^{-\beta}}{\Gamma(1-\beta)}=\partial^\beta_t p(x,t),
\end{equation}
we can also write the governing equation \eqref{CTRWgov} in the form
\begin{equation}\label{RLtoC}
{\mathbb D}_t^\beta p(x,t) -p(x,0)\frac{t^{-\beta}}{\Gamma(1-\beta)}=A\ p(x,t) .
\end{equation}
The Laplace symbol $f(\phi)=\phi^\beta$ can be computed directly from \eqref{defbern} with $a=b=0$ and $\nu(dt)=\beta t^{-\beta-1}dt/\Gamma(1-\beta)$ using integration by parts and the definition of the Gamma function \cite[Proposition 3.10]{FCbook}.  Then $\bar\nu(t)=t^{-\beta}/\Gamma(1-\beta)$ and the operator $\mathcal{D}^f$ of Toaldo \cite{toaldopota} defined by \eqref{28} reduces to the Riemann-Liouville fractional derivative.  Similarly, the operator $\mathfrak{D}_t^f$ of \cite{toaldopota} defined by \eqref{29} reduces to the \bt{Caputo} fractional derivative.  Then the governing equation \eqref{general} of Toaldo \cite{toaldopota} with $g(t)=p(x,t)$ reduces to \eqref{RLtoC}.  Since the conjugate Bernstein function $f^\star(\phi)=\phi/f(\phi)=\phi^{1-\beta}$, the same calculation as for $f$ shows that $\bar\nu^\star(t)=t^{\beta-1}/\Gamma(\beta)$, and then the operator $\Phi_t$ of Magdziarz \rev{and Schilling} \cite{magda} defined by \eqref{21schi} with $M(t)=\bar\nu^\star(t)$ is the Riemann-Liouville fractional integral
\begin{equation*}
{\mathbb I}_t^\beta p(x,t)=\frac{1}{\Gamma(\beta)}\int_0^t p(x,t-u)u^{\beta-1}du=\frac{1}{\Gamma(\beta)}\int_0^t p(x,u)(t-u)^{\beta-1}du
\end{equation*}
of order $\beta$ \cite[p.\ 250]{FCbook}.   Then the governing equation \eqref{generaladj} of Magdziarz \cite{magda} with $g(t)=p(x,t)$ reduces to
\begin{align}
\frac{d}{dt}  p(x,t) \,  = \, \frac{1}{\Gamma(\beta)} \frac{d}{dt} \int_0^t A p(x,s) \, (t-s)^{\beta-1} \, ds =\frac{d}{dt} A\, {\mathbb I}_t^\beta p(x,t).
\label{generaladjEx1}
\end{align}
Since $\frac{d}{dt} {\mathbb I}_t^\beta={\mathbb D}_t^{1-\beta}$ we can also write \eqref{generaladjEx1} in the form
\[ \frac{d}{dt}  p(x,t) \,  = \, {\mathbb D}_t^{1-\beta} A\, p(x,t)= A\,  {\mathbb D}_t^{1-\beta} p(x,t),\]
which is commonly seen in applications \cite{HLS10PRL,MBK99}.  The heuristic derivation of \eqref{generaladjEx1} is to simply apply the operator ${\mathbb D}_t^{1-\beta}$ to both sides of \eqref{RLtoC}, or the equivalent form $\partial_t^\beta p=Ap$, but the initial condition requires some care.
\end{ex}

\begin{ex}\normalfont
Replacing the iid jumps $X_n$ in Example \ref{Sec4Ex1} with a convergent triangular array, we obtain a limit $B(L^f_t)$ where $B(t)$ is an arbitrary L\'evy process \cite[Theorem 3.6]{meertri} with generator $A$ given by the L\'evy-Khintchine formula \cite[Theorem 3.1.11]{RVbook}.  Then all the results of Example \ref{Sec4Ex1} hold with $A= D\partial^2_x$ replaced by this L\'evy generator, in ${\mathbb R}^1$ or in ${\mathbb R}^d$ for any finite dimension $d$ \cite[Theorem 4.1]{meertri}.  The same is true more generally for Markov generators $A$, where the jump distribution depends on the current state \cite[Theorem 4.2]{KoloCTRW}.
\end{ex}

\begin{ex}\normalfont
For a CTRW with deterministic particle jumps $X_n=1$, and the same waiting times $W_n$ as in Example \ref{Sec4Ex1}, we have $S(n)=n$ and the CTRW $S(N_t)=N_t$ converges to the inverse stable subordinator: $n^{-\beta}N_{nt}\Rightarrow L^f_t$.  The probability density $l(x,t)$ of $L^f_t$ solves the time-fractional equation \eqref{CTRWgov} with $\mathbb{C}_f(\partial_t)=\partial_t^\beta$ and $A= -\partial_x$ \cite[Eq.\ (5.7)]{hittingTime}.  Several equivalent governing equations for the inverse stable subordinator $L^f_t$ are also discussed in \cite{hittingTime}, including equation \eqref{general} of Toaldo \cite{toaldopota} (see \cite[Eq.\ (5.9)]{hittingTime}) and equation \eqref{generaladj} of Magdziarz  \cite{magda} (see \cite[Eq.\ (5.18)]{hittingTime}) with $g(t)=l(x,t)$ and  $A= -\partial_x$. In this case, the moment generating function \eqref{qDef} can be written explicitly as $q(\lambda,t)=E_\beta(\lambda t^\beta)$
in terms of the Mittag-Leffler function
\begin{equation}\label{MittegLefflerFcn}
E_\beta(z)=\sum_{j=0}^\infty \frac{z^j}{\Gamma(1+\beta j)},
\end{equation}
for any $\lambda\leq 0$, see Bingham \cite{Bingham}.  The function $q(\lambda,t)$ solves \eqref{problrel}, which can be rewritten in this case as
\[{\mathbb D}_t^\beta q(\lambda,t)-\frac{t^{-\beta}}{\Gamma(1-\beta)}=\lambda q(\lambda,t) ,\]
or equivalently, using \eqref{RLtoCaupto}, as
\[\partial_t^\beta q(\lambda,t)=\lambda q(\lambda,t) .\]
That is, the moment generating function of the inverse stable subordinator is an eigenfunction of the Caputo fractional derivative.  The function $q(\lambda,t)$ also solves \eqref{problreladj}, which can be rewritten in this case as
\[ \frac{d}{dt}  q(\lambda,t) = \lambda\,  {\mathbb D}_t^{1-\beta} q(\lambda,t).\]
Here $f(\phi)=\phi^\beta$ varies regularly at zero with index $\beta$, $\bar\nu(t)=t^{-\beta}/\Gamma(1-\beta)$ varies regularly at infinity with index $-\beta$, and Item (\ref{3rel}) of Theorem \ref{te11} shows that $t\mapsto q(\lambda,t)$ also varies regularly at infinity with index $-\beta$, with $q(\lambda,t)\sim \lambda^{-1}t^{-\beta}/\Gamma(1-\beta)$ as $t\to\infty$, compare Scalas \cite[Eq.\ (24)]{Scalas2006Lecture}.  The potential measure $U^{\sigma^f}(dt)$ has Laplace-Stieltjes transform $1/f(\phi)=\phi^{-\beta}$ \cite[Eq.\ (5.12)]{librobern}, and inverting \cite[Eq.\ (2.25)]{FCbook} shows that \eqref{potmeas} holds with $c=0$ and $u_{\sigma^f}(t)=t^{\beta-1}/\Gamma(\beta)$. Then it follows from \cite[Thm 10.3]{librobern} that $f$ is a special Bernstein function.  In fact $f^\star(\phi)=\phi^{1-\beta}$ from \eqref{defbern*} with $a^\star=b^\star=0$ and $\nu^\star(dt)=(1-\beta) t^{\beta-2}dt/\Gamma(\beta)$ concentrated on $t>0$.  Note also that \cite[Eq.\ (10.9)]{librobern} $u_{\sigma^f}(t)=\bar\nu^\star(t)$.
\end{ex}

\begin{ex}\label{ExKTS}\normalfont
Let $\nu(dt)=\beta t^{-\beta-1}e^{-ct}dt/\Gamma(1-\beta)$ for some $0<\beta<1$ and $c>0$ and compute
\begin{equation}\label{ExKTSeq1}
f(\phi)=c^\beta+\int_0^\infty\left(1-e^{-\phi t}\right)\nu(dt)=(\phi+c)^\beta
\end{equation}
using integration by parts, compare \cite[Eq.\ (7.9)]{FCbook}.  Then $\sigma^f_t$ is a tempered stable subordinator \cite[Section 7.2]{FCbook} killed at rate $c^\beta$: If $D(t)$ is a tempered stable subordinator with ${\mathbb E}[e^{-sD(t)}]=\exp[-t\{(\phi+c)^\beta-c^\beta\}]$ and $S$ is an exponential random variable independent of $D(t)$ with ${\mathbb P}[S>t]=\exp[-c^\beta t]$, then we can let
\[\sigma^f_t=\begin{cases} D(t), & \text{if }t<S,\\ \infty,& \text{if }t>S,\end{cases} \]
and it is easy to check that $f$ is the Laplace symbol of this process. Now
\begin{align}
\bar\nu(t) \, = \, \bt{ a+ \nu(t, \infty) \, = \, c^\beta +} \, \frac{\beta}{\Gamma(1-\beta)}\int_t^\infty s^{-\beta-1}e^{-cs}ds
\label{incgamma}
\end{align}
involves the incomplete Gamma function, which cannot be written in closed form.  It is well known that $\phi^\beta\widetilde u(\phi)$ is the Laplace transform of ${\mathbb D}_t^\beta u(t)$ \cite[p.\ 39]{FCbook}.  Using the shift property $\widetilde u(\phi+c)={\mathcal L}[e^{-ct}u(t)]$ twice, it follows that the tempered fractional derivative ${\mathbb D}_t^{\beta,c} u(t)=e^{-c t}{\mathbb D}^\beta_t (e^{ct}u(t))$ has Laplace symbol $f$, i.e.,
${\mathcal L}[{\mathbb D}_t^{\beta,c} u(t)]=(\phi+c)^\beta \widetilde u(\phi),$
see \cite[p.\ 209]{FCbook}.  \MMM{If $u,u'$ are in $L^1({\mathbb R})$ then one can write the tempered fractional derivative explicitly as \cite[Theorem 2.9]{TFI}
\[
{\mathbb{D}}^{\beta,c}_{+}u(t)={c}^{\beta}f(t)+\frac{\beta}{\Gamma(1-\beta)}\int_{-\infty}^{t}\frac{f(t)-f(u)}{(t-u)^{\beta+1}}
\,e^{-c(t-u)}du.
\]}
It follows from \eqref{eq2.18mmm} that
$\mathcal{D}^f u(t)={\mathbb D}_t^{\beta,c} u(t),$
and hence
the governing equation \eqref{general} of Toaldo can be written in the form
\[{\mathbb D}_t^{\beta,c} g(t) - \bar{\nu}(t) g(0)  \, = \, A g(t) .\]
Recall that ${\mathcal L}[t^{\MMM{\beta-1}}/\Gamma(\beta)]=\phi^{-\beta}$, and use the shift property again, along with \eqref{Mdef}, to see that
\[M(t)={\mathcal L}^{-1}[(\phi+c)^{-\beta}]=\frac{1}{\Gamma(\beta)}t^{\MMM{\beta-1}}e^{-ct} .\]
Then the operator $\Phi_t$ of Magdziarz and Schilling \cite{magda} defined by \eqref{21schi} can be written as
\[\Phi_tu(t) \, = \, \frac{d}{dt} \frac{1}{\Gamma(\beta)} \int_0^t u(s) M(t-s) ds=\frac{d}{dt} {\mathbb I}^{\beta,c}_t u(t), \]
where the tempered fractional integral \cite{TFI} \rev{given by}
\[{\mathbb I}^{\beta,c}_t u(t):= \frac{1}{\Gamma(\beta)} \int_0^t u(t-s) s^{\MMM{\beta-1}}e^{-cs} ds,\]
\rev{appears. Hence} the governing equation \eqref{generaladj} of Magdziarz \rev{and Schilling} \cite{magda} reduces to
\begin{align}
\frac{d}{dt}  g(t) \,  = \, \frac{1}{\Gamma(\beta)} \frac{d}{dt} \int_0^t A g(t-s) \, s^{\beta-1} e^{-cs}\, ds =\frac{d}{dt} A\, {\mathbb I}_t^{\beta,c} g(t).
\label{generaladjExTFK}
\end{align}
The potential measure $U^{\sigma^f}(dt)$ has Laplace-Stieltjes transform $1/f(\phi)=(\phi+c)^{-\beta}$ \cite[Eq.\ (5.12)]{librobern}, and inverting shows that \eqref{potmeas} holds with $c=0$ and $u_{\sigma^f}(t)=M(t)=\bar\nu^\star(t)=t^{\beta-1}e^{-ct}/\Gamma(\beta)$. Then it follows from \cite[Thm 10.3]{librobern} that $f$ is a special Bernstein function.  It is easy to check that $f(\phi)=(\phi+c)^\beta$ varies regularly at zero with index $\beta=0$, and then Item (\ref{3rel}) of Theorem \ref{te11} shows that both $\bar\nu(t)$ and $t\mapsto q(\lambda,t)$ are slowly varying at infinity.
\end{ex}

\begin{ex}\label{ExTS}\normalfont
The (unkilled) tempered stable subordinator $D(t)$ related to the Bernstein function $\l c+\phi \r^\beta - c^\beta$ is also included in our framework. The L\'evy density in this case is
\begin{align}
\nu(t) \, = \, \frac{\beta t^{-\beta-1}e^{-ct}}{\Gamma(1-\beta)} \, = \, \frac{\beta}{\Gamma(1-\beta)} \, e^{-ct } \, \int_0^\infty e^{-st} \frac{s^\beta}{\Gamma(1+\beta)} ds .
\label{prodcm}
\end{align}
Since \eqref{prodcm} is the product of two completely monotone functions, it is completely monotone \cite[Corollary 1.6]{librobern}. This means that
\begin{align}
\phi \mapsto (\phi + c)^\beta - c^\beta \, = \, \int_0^\infty \l 1-e^{-\phi t} \r \frac{\beta t^{-\beta-1}e^{-ct}}{\Gamma(1-\beta)} dt
\end{align}
is a complete Bernstein function by \cite[Def 6.1]{librobern} and therefore it is also special by \cite[Prop 7.1]{librobern}. The tail of the L\'evy measure here is similar to \eqref{incgamma}, but without the constant term:
\begin{align}
\bar{\nu}(t) \, = \, \frac{\beta}{\Gamma (1-\beta)} \int_t^\infty s^{-\beta-1}e^{-cs} ds .
\end{align}
Now \eqref{Mdef} becomes
\begin{align}
M(t) \, = \, \mathcal{L}^{-1} \left[ \l \l c+\phi \r^\beta - c^\beta \r^{-1}  \right] (t),
\end{align}
\MMM{which seems difficult to invert in closed form.
Since $a^\star =0$ in view of \eqref{asbs}, we get from \eqref{213} that
\begin{align}
\bar{\nu}^\star (t) \, = \,\nu^\star (t, \infty) \, = u_{\sigma^f}(t)
\end{align}
where $u_{\sigma^f}(t)$ is the potential density of the tempered stable subordinator.  Since $f$ is special, its conjugate $f^\star(\phi)=\phi[(\phi + c)^\beta - c^\beta]^{-1}$ is the Laplace symbol of some subordinator $\sigma^\star(t)$.}

\end{ex}

\begin{ex} \normalfont
\label{exdo}
Distributed order fractional derivatives are also included in our framework.  Let $(0,1) \ni y \mapsto  \alpha (y)$ be a function strictly between zero and one and let $p(\cdot)$ be a measure on $(0,1)$. Choose $\alpha(y)$ and $p$ in such a way that, for $s>0$, it is true that
\begin{align}
&\int_0^\infty (s \wedge 1) \int_0^1 \frac{\alpha (y) s^{-\alpha (y)-1}}{\Gamma (1-\alpha(y))} \, p(dy)\, ds \notag \\
 = \,& \int_0^1 \int_0^\infty (s \wedge 1) \, \frac{\alpha (y) s^{-\alpha (y)-1}}{\Gamma (1-\alpha(y))} \, ds \,  p(dy) < \, \infty.
\label{dolev}
\end{align}
Under \eqref{dolev}
\begin{align}
\MMM{\nu(s)=}\ \int_0^1\frac{\alpha (y) s^{-\alpha (y)-1}}{\Gamma (1-\alpha(y))} \, p(dy)
\label{levdo}
\end{align}
is a L\'evy \MMM{density} and therefore
\begin{align}
\MMM{f(\phi)}\,= \, & \int_0^\infty \l 1-e^{-s\phi} \r \, \nu(ds) = \int_0^1 \phi^{\alpha(y)} p(dy)
\label{berdo}
\end{align}
is a Bernstein function. The operator $\mathfrak{D}_{\rev{t}}^f$ corresponding to \eqref{berdo} may be viewed as a distributed order fractional derivative \cite[Remark 4.3]{toaldodo} since here
\begin{align}
\MMM{\bar\nu(s) =}\ \int_0^1 \frac{s^{-\alpha(y)}}{\Gamma (1-\alpha (y))}p(dy)
\end{align}
and therefore
\begin{align}
&\mathfrak{D}_{\rev{t}}^fu(t) \notag \\ = \, &\frac{d}{dt} \int_0^t u(s) \, \int_0^1 \frac{(t-s)^{-\alpha (y)}}{\Gamma (1-\alpha(y))} \, p(dy)\,  \, ds- u(0) \int_0^1 \frac{t^{-\alpha (y)}}{\Gamma (1-\alpha(y))} \, p(dy) \notag \\
 = \, &\int_0^1 \frac{1}{\Gamma(1-\alpha(y))} \frac{d}{dt} \int_0^t u(s) (t-s)^{-\alpha(y)} ds \, p(dy) - u(0) \int_0^1 \frac{t^{-\alpha (y)}}{\Gamma (1-\alpha(y))} \, p(dy) \notag \\
 = \, & \int_0^1 \mathbb{D}_t^{\beta(y)} u(t)p(dy) \, - \,  u(0) \int_0^1 \frac{t^{-\alpha (y)}}{\Gamma (1-\alpha(y))} \, p(dy).
\end{align}
\MMM{The traditional form of the distributed order derivative is the special case $\alpha (y) = y$.}
Now note that the function $t \mapsto t^{-\alpha(y)}$ is completely monotone for each fixed $y$ since it is the Laplace transform of the measure $s^{\alpha(y)-1}ds/\Gamma(\alpha(y))$. Therefore, when $\alpha(y)$ and $p(\cdot)$ are such that \eqref{dolev} is fulfilled, also the function
\begin{align}
\MMM{\bar\nu(t)} \, = \, & \int_0^1 \frac{1}{\Gamma(1-\alpha(y))}\int_0^\infty e^{-ts} \frac{s^{\alpha(y)-1}}{\Gamma(\alpha(y))} \, ds \, p(dy) \, \notag \\
= \, &  \int_0^\infty e^{-ts} \int_0^1 \frac{1}{\Gamma(1-\alpha(y))} \frac{s^{\alpha(y)-1}}{\Gamma(\alpha(y))} \,p(dy) \, ds
\label{412}
\end{align}
is completely monotone. Therefore \eqref{412} is the tail of a L\'evy measure with a completely monotone density by Remark \ref{rem22}. This implies that \MMM{$f$}
is a complete Bernstein function \cite[Def 6.1]{librobern} and therefore it is also special \cite[Prop. 7.1]{librobern}. Hence
\begin{align}
\MMM{f^\star(\phi)=}\ \frac{\phi}{\int_0^1 \phi^{\alpha(y)}p(dy)}
\end{align}
is a (special) Bernstein function, \MMM{and there exists a subordinator $\sigma^\star(t)$ with Fourier symbol $f^\star$.}  \MMM{From \eqref{213} we have
\begin{align}
\bar{\nu}^\star (t) \, = \,\nu^\star (t, \infty) \, = u_{\sigma^p}(t),
\end{align}
where $u_{\sigma^p}(t)$ is the potential density of the subordinator with Laplace exponent \eqref{berdo}.}
Then observe that Items (1), (2) and (3) of Theorem \ref{te11} apply to this case. In particular Item (3) here is in accordance with \cite[Thm 2.3]{kochudo}. Item (4) applies if $p$ and $\alpha$ are such that $f(\phi)=\int_0^1 \phi^{\alpha (y)}p(dy)$ is regularly varying at $0+$.  For example, if $\alpha(y)=y$ and $p(dy)=p_0(y)dy$ where $p_0$ is regularly varying at zero with some index $\gamma>-1$, then $f$ is slowly varying \cite[Lemma 3.1]{ultraslow}, and hence $\bar\nu(t)$ and $q(\lambda,t)$ are slowly varying at $t=\infty$.  This is a model for ultraslow diffusion \cite{ultraslow} where a plume of particles spreads at a logarithmic rate in time. Here the kernel of Magdziarz and Schilling \cite{magda} can be computed from
\begin{align}
\MMM{M(t)=}\ \mathcal{L}^{-1} \left[ \frac{1}{\int_0^1 \phi^{\alpha(y)}p(dy) }\right] (t).
\label{magopdo}
\end{align}
In the special case $f(\phi)=p_1\phi^{\beta_1}+p_2\phi^{\beta_2}$ (retarding subdiffusion) one can write
\[M(t)=p_2^{-1}t^{\beta_2}E_{\rev{\beta_2-\beta_1},\beta_2+1}\left(-\frac{p_1}{p_2} t^{\beta_2-\beta_1}\right) ,\]
where the two parameter Mittag-Leffler function
\[E_{s,t}(z)=\sum_{n=0}^\infty \frac{z^n}{\Gamma(sn+t)} ,\]
compare Chechkin et al. \cite[Eq.\ (16)]{retarding}.
\end{ex}

\frev{
\section*{Acknowledgements}
The authors would like to thank Vassili Kolokoltsov and Lorenzo Toniazzi, University of Warwick, for helpful discussions.  We would also like to thank the referees for useful comments. \\
M. M. Meerschaert was partially supported by ARO grant W911NF-15-1-0562 and NSF grants EAR-1344280 and DMS-1462156
}

\vspace{1cm}

\end{document}